\documentclass[a4paper,12pt]{article}
\usepackage{amsmath,amsfonts,amssymb,amsthm,amscd}
\usepackage[arrow, matrix, curve]{xy}
\usepackage[colorlinks=true,citecolor=blue]{hyperref}

\font\cmssl=cmss10 at 12 pt

\usepackage{slashed}
\usepackage[normalem]{ulem}

\newtheorem{thm}{Theorem}
\newtheorem{lem}[thm]{Lemma}
\newtheorem{prop}[thm]{Proposition}
\newtheorem{defn}[thm]{Definition}
\newtheorem{cor}[thm]{Corollary}
\newtheorem{rem}[thm]{Remark}
\newtheorem{notation}[thm]{Notation}
\newtheorem{exa}[thm]{Example}

\title{Components of generalised complex structures on transitive Courant algebroids}

\date{\today}

\author{Vicente Cort\'es and  Liana David}

\begin{document}

\maketitle
\abstract{Generalised almost complex structures $\mathcal J$ on transitive Courant algebroids $E$ are studied in terms of their 
components with respect to a splitting $E\cong TM \oplus T^*M \oplus \mathcal G$, where $M$ denotes the base
of $E$ and $\mathcal G$ its bundle of quadratic Lie algebras. Necessary and sufficient integrability equations for $\mathcal J$ are established in this formalism. 
As an application, it is shown that the integrability of $\mathcal J$ implies that one of the components defines a Poisson structure on $M$.
Then the structure (normal form) of generalised complex structures for which the Poisson structure is non-degenerate is determined. 
It is shown that it is fully encoded in a pair $(\omega , \rho )$ consisting of a symplectic structure $\omega$ on $M$ and 
a representation $\rho : \pi_1(M) \to \mathrm{Aut}(\mathfrak g, \langle \cdot ,\cdot \rangle_{\mathfrak{g}}, J_{\mathfrak{g}})$
by automorphism of a quadratic Lie algebra $(\mathfrak g, \langle \cdot ,\cdot \rangle_{\mathfrak{g}})$ commuting with an 
integrable (in the sense of Lie algebras) skew-symmetric complex structure $J_{\mathfrak{g}}$. Examples of such representations and obstructions for the existence of non-degenerate generalised complex structures  are discussed.   Finally, a construction of  generalised complex structures 
on transitive Courant algebroids over complex manifolds for which the Poisson structure degenerates along a complex analytic hypersurface is presented.  

\medskip\noindent
{\it MSc classification:} 53D18 (Generalized geometry a la Hitchin); 53D17 (Poisson manifolds) 

}

\tableofcontents
\section{Introduction}

As is well-known, the notion of a generalised complex structure on a manifold $M$ was proposed by Hitchin \cite{H} as a concept unifying 
complex and symplectic structures on $M$.  For this purpose, the tangent bundle $TM$ was replaced 
by the generalised tangent bundle $\mathbb{T}M$, which carries the structure of an exact Courant algebroid on which 
the generalised complex structure is realised as an endomorphism. Generalised complex structures on exact Courant algebroids 
were studied in \cite{gualtieri-thesis,gualtieri-annals,crainic} and many other works. 

Contrary to exact Courant algebroids, transitive Courant algebroids $E$ as classified in \cite{chen}  incorporate a bundle of quadratic Lie algebras $\mathcal G$ (such as 
the adjoint bundle of a principle bundle). In fact, they admit a dissection
\begin{equation}\label{standard:eq} E\cong TM \oplus T^*M \oplus \mathcal G,\end{equation} 
that is, an isomorphism 
to a standard Courant algebroid. The additional bundle $\mathcal G$ and the entailed geometric structures 
make the study of geometric structures on such  Courant algebroids richer and especially interesting from the point of heterotic supergravity and string theory \cite{GF_Rubio_Tipler}. 

Recently, a Darboux theorem for (regular) generalised complex structures on transitive Courant algebroids was obtained in \cite{darboux}, generalising Gualtieri's Darboux 
theorem  \cite{gualtieri-thesis,gualtieri-annals}. 

In the present article 
we continue the study of generalised complex structures $\mathcal J$ on transitive Courant algeboids $E\to M$ by analysing the components of 
$\mathcal J$ with respect to the decomposition \eqref{standard:eq} and deriving various structural results.

Section \ref{prelim:sec} is intended to collect basic facts and fix notation.

As a first step,  in Section \ref{components-sect} 
we characterise generalised almost complex structures $\mathcal J$  in terms of the tensor fields $(A,B,C,J, \mu, \nu)$ appearing as components and their algebraic relations (see Lemma~\ref{comp:lem}). 
The more involved system of partial differential equations for the component fields which expresses the integrability of $\mathcal J$ is obtained in Theorem \ref{integr-thm}.  The strength of the system is illustrated   in Proposition \ref{poisson:prop} by proving 
that the component $B : T^*M \to TM$ of $\mathcal J$ always defines a Poisson structure.  
The result is well-known in the exact case \cite{Hi_2006}.  
The proof of Proposition \ref{poisson:prop} uses only one of the  10 equations of the system. 

In Section \ref{2nd_data:sec}, the components $(A,B,C,J, \mu, \nu)$ are related to a different system of data $(W, \mathcal D,  \sigma, \varepsilon )$ encoding a   generalised almost 
complex structure. The data  $(W, \mathcal D,  \sigma, \varepsilon )$ were introduced in \cite{darboux} and consist of a pair of bundles and a pair of  tensor fields. A precise dictionary is 
provided in Theorem \ref{comp_to_data:thm}. In this way, results obtained in \cite{darboux} can be conveniently translated into the components formalism.

The above dictionary is used in Section \ref{nondegGCS:sec} to prove a structure theorem for  non-degenerate generalised complex structures, where  a 
generalised almost complex structure is called  
non-degenerate if the corresponding tensor field $B$ is non-degenerate. We  describe non-degenerate generalised almost complex structures in Proposition \ref{nondeg-alg} and the subclass of integrable ones in Proposition \ref{intBsymp:thm}. Then we show in Theorem \ref{class_nondeg:thm} that, up to isomorphisms, the underlying 
Courant algebroid is an untwisted  Courant algebroid (as defined in Section \ref{prelim:sec}) and $\mathcal J$ takes the form 
$\mathcal J_\omega \oplus A$, where $\mathcal J_\omega$ is the generalised complex structure on $TM\oplus T^*M$ associated with the
symplectic structure $\omega = -B^{-1}$ and $A$ is a parallel field of skew-symmetric fiber-wise complex structures on the flat bundle $\mathcal G$. 
The following classification in terms of representations of the fundamental group of $M$
 (together with further details) 
 is given in Theorem \ref{classif:cor}:

\medskip 
\noindent 
{\bf Theorem} {\it Let $(M, \omega )$ be a (connected)  symplectic manifold and\linebreak[4]  $(\mathfrak g , \langle \cdot , \cdot \rangle_{\mathfrak{g}})$ a quadratic Lie algebra.  
There is a natural bijection between, on the one hand, isomorphism classes  of non-degenerate generalised complex structures with underlying symplectic structure $\omega$ on transitive Courant algebroids $E$ over $M$ with quadratic Lie algebra bundle of fiber type $(\mathfrak g , \langle \cdot , \cdot \rangle_{\mathfrak{g}})$ and, on the other hand,  isomorphism classes  of pairs $(J_\mathfrak{g}, \rho )$, where $J_\mathfrak{g}$ 
is a skew-symmetric integrable complex structure on $\mathfrak g$ and 
$\rho : \pi_1(M) \to \mathrm{Aut}(\mathfrak{g},\langle \cdot , \cdot \rangle_{\mathfrak{g}}, J_\mathfrak{g})$ is a representation by $J_{\mathfrak{g}}$-linear automorphisms of the quadratic Lie algebra.}

\medskip\noindent
 For illustration,  the construction is specialised  in  various examples  providing additional information.  In particular, we present examples of non-degenerate generalised  complex structures on transitive Courant algebroids of  heterotic and non-heterotic type and examples of quadratic Lie algebras which cannot be the fiber type
of a transitive Courant algebroid  admitting a non-degenerate generalised complex structure.

\medskip
Finally, 
in Section \ref{deg:Sec} we present a class of generalised  complex structures associated with 
a complex manifold $(M,J)$ endowed with a holomorphic Poisson structure $\beta$ and a representation $\rho$ of  the fundamental group $\pi_1(M)$ 
by automorphisms of a quadratic Lie algebra endowed with an integrable skew-symmetric complex structure
 (see Proposition~\ref{poisson_ex:prop}). 
Such  generalised  complex structures  are non-degenerate on the complement of the complex analytic hypersurface $\{ \det \beta =0\} \subset M$. 
This includes examples with compact base manifold $M$ and infinite holonomy $\rho (\pi_1(M))$.

\medskip\noindent 
{\bf Acknowledgements.} Research of V.C.\ is funded by the Deutsche Forschungsgemeinschaft (DFG, German Research Foundation) under SFB-Gesch\"aftszeichen 1624 -- Projektnummer 506632645 and under Germany's Excellence Strategy, EXC 2121``Quantum Universe,''\linebreak[4] 
390833306.  L.D.\ thanks University of Hamburg for hospitality, excellent working conditions and financial support during her visit in 
June 2025.

\section{Preliminary material}
\label{prelim:sec}
For completeness of our exposition and to fix notation we now recall basic facts 
on  standard  Courant algebroids. We assume that the reader has some familiarity with the theory of Courant algebroids. 
For more details, see e.g.\  \cite{chen,cortes-david-JSG, darboux}. We follow the conventions we used in our previous works  \cite{cortes-david-JSG} and \cite{darboux}.\

\begin{notation}{\rm For two forms $\alpha \in \Omega^{p}(M, \mathcal G)$ and  $\beta \in \Omega^{q} (M, \mathcal G)$ with values in a vector bundle  $\mathcal G\rightarrow M$ with scalar product 
$\langle \cdot , \cdot \rangle_{\mathcal G}$, we denote by  $\langle \alpha \wedge  \beta \rangle_{\mathcal G}  \in \Omega^{p+q}(M, \mathcal G)$ 
the $(p+q)$-form with values in $\mathcal G$ obtained by combining the exterior product of (scalar-valued)  forms on $M$ with the scalar product
of $\mathcal G .$ For the exterior product of  forms on $M$ we use the convention determined by the recursive relation
 $$
i_{X} (\alpha \wedge \beta ) =  (i_{X} \alpha )\wedge \beta + (-1)^{p} \alpha \wedge i_{X} \beta ,
$$
for any $\alpha\in \Omega^{p}(M)$ and $\beta \in \Omega^{q}(M)$, where $i_{X} :\Omega^{*} (M) \rightarrow \Omega^{* -1} (M)$ is 
the interior product by $X\in  {\mathfrak X}(M)$, defined by $(i_{X} \omega )(X_{1}, \cdots , X_{p-1}) := \omega  (X, X_{1}, \cdots , X_{p-1})$
for any $\omega\in \Omega^{p}(M)$ and $X_{i} \in {\mathfrak X}(M).$ In particular, 
$(\alpha \wedge \beta )(X, Y) = \alpha (X) \beta (Y) -\alpha  (Y)\beta (X)$, for any $\alpha , \beta \in \Omega^{1}(M)$,
$$
(\omega_{1} \wedge \omega_{2} )(X, Y, Z) = \sum_{(X:Y:Z)} \omega_{1} (X)\omega_{2}(Y, Z)
$$
and 
$$
(\omega_{2} \wedge \omega_{2} )(X, Y, Z, V) = 2\sum_{(X:Y:Z)} \omega_{2} (X, Y) \omega_{2} (Z, V)
$$
for any $\omega_{i}\in \Omega^{i}(M)$
($i=1,2$)   where $\sum_{(X:Y:Z)}$ denotes sum over cyclic permutations.\

We shall often use the same notation for a linear map between real vector spaces  and its complex linear extension to the complexified 
vector spaces.  Unless otherwise stated, we only consider Courant algebroids of neutral signature. All manifolds are assumed connected. 
}
\end{notation}

Recall that any standard Courant algebroid $E= \mathbb{T}M \oplus \mathcal G$, where $\mathbb{T}M = TM\oplus T^{*}M$,   is defined by a  quadratic Lie algebra bundle
$(\mathcal G , [\cdot , \cdot ]_{\mathcal G}, \langle \cdot , \cdot \rangle_{\mathcal G})$  and  data
$(\nabla , R, H)$, where $\nabla$ is a connection on $\mathcal G$   which  preserves   $[\cdot , \cdot ]_{\mathcal G}$ and 
$\langle \cdot , \cdot \rangle_{\mathcal G}$, 
\begin{equation}\label{def-data}
R^{\nabla}(X, Y) r = [R(X, Y), r]_{\mathcal G},\ d^{\nabla}R =0,\ d H = \langle R\wedge R\rangle_{\mathcal G},
\end{equation}
for any $X, Y\in {\mathfrak  X} (M)$, $r\in \Gamma (\mathcal G)$ and  $R^{\nabla}$ denotes the  curvature of $\nabla$.
The anchor of $E$ is the natural projection  $\pi : E\rightarrow TM$ and its  scalar product is given by
\begin{equation}
\langle X +\xi + r, Y+\eta + \tilde{r} \rangle = \frac{1}{2} ( \eta (X) +\xi (Y)) + \langle r, \tilde{r} \rangle_{\mathcal G}
\end{equation}
where $X+\xi  , Y+\eta\in \mathbb{T}M$ and $r, \tilde{r} \in \mathcal G.$
The Dorfman bracket  $[\cdot , \cdot ] : \Gamma (E)\times \Gamma (E) \rightarrow \Gamma (E)$ is given by
\begin{align}
\nonumber& [X, Y] = \mathcal L_{X}Y + i_{Y} i_{X} H + R(X, Y)\\
\nonumber& [ X, r] =  - 2\langle i_{X}R, r\rangle_{\mathcal G} + \nabla_{X} r \\
\nonumber& [r_{1}, r_{2} ] = 2 \langle \nabla r_{1}, r_{2}\rangle_{\mathcal G} +  [r_{1}, r_{2} ]_{\mathcal G} \\
\label{dorfman1}& [X, \eta ] = \mathcal L_{X} \eta ,\ [\eta_{1}, \eta_{2} ] = [ r, \eta ] =0,
\end{align}
for any $X , Y\in {\mathfrak X}(M)$,   $\eta , \eta_{1}, \eta_{2} \in \Omega^{1}(M)$ and $r, r_{1} r_{2} \in \Gamma (\mathcal G)$, 
together with 
\begin{equation}\label{dorfman2}
[u, v] + [v, u] = 2 d \langle u, v\rangle ,\ \forall u, v\in \Gamma (E).
\end{equation} 
The natural projections of a standard Courant algebroid $E=\mathbb{T}M \oplus \mathcal G$ to the direct summands $T^*M$ and $\mathcal G$ will be denoted by $\pi_{T^*M}$ and $\pi_{\mathcal G}$, respectively.\

The parallel transport defined by $\nabla$ induces  isomorphisms between the fibers $(\mathcal G, [\cdot , \cdot ]_{\mathcal G},
\langle \cdot , \cdot \rangle_{\mathcal G})\vert_{p}$, $p\in M$, and a fixed quadratic Lie algebra $(\mathfrak{g}, \langle \cdot , \cdot \rangle_{\mathfrak{g}})$, called the fiber type of $\mathcal G.$\

By an untwisted Courant algebroid we mean a standard Courant algebroid  $E = \mathbb{T}M\oplus \mathcal G$ for which $R = 0$ and  $H =0$. In particular,   
the connection $\nabla$ is flat and, if $M$ is simply connected,  then $\mathcal G = M\times \mathfrak{g}$ is the trivial vector bundle with fiber a 
quadratic Lie algebra. 
By a  heterotic Courant algebroid 
we mean a transitive Courant algebroid  for which 
the quadratic Lie algebra bundle is the adjoint bundle of a principal bundle (with the connection induced from
a principal connection).\

As proved in \cite{chen}, any
transitive Courant algebroid is isomorphic to a standard Courant algebroid.\
Here we follow the standard terminology according to which an isomorphism $I: E\to E'$ 
of Courant algebroids over $M$  is an isomorphism of the underlying vector bundles 
$\mathrm{pr} : E\to M$ and $\mathrm{pr}' : E' \to M$ intertwining the anchors, scalar products and  Dorfman brackets. 
Isomorphisms of vector bundle are fiber-preserving in the sense that $\mathrm{pr}' \circ I = \mathrm{pr}$. 
More generally, one can consider equivalences of  vector bundles and Courant algebroids, which corresponds to weakening 
the latter equation to  $\mathrm{pr}' \circ I = \varphi \circ \mathrm{pr}$, where $\varphi$ is a diffeomorphism of $M$. 

A  Courant algebroid isomorphism $I : E_{1} \rightarrow E_{2}$ between two standard Courant algebroids 
over a manifold $M$, with 
quadratic Lie algebra bundles $(\mathcal G_{i}, [\cdot  , \cdot ]_{\mathcal G_{i}}, \langle \cdot , \cdot \rangle_{\mathcal G_{i}} )$ and
defining  data  $(\nabla^{i}, R_{i}, H_{i})$,  is determined by a system $(K, \Phi  , \beta)$ where
$K\in \mathrm{Isom} (\mathcal G_{1}, \mathcal G_{2})$ is an isomorphism of quadratic Lie algebra bundles, 
$\Phi \in \Omega^{1}(M, \mathcal G_{2})$,  $\beta \in \Omega^{2}(M)$, 
as follows:
\begin{align}
\nonumber& I(X) = X+  i_{X}\beta - \Phi^{*}\Phi  (X) +\Phi (X) \\
\nonumber& I(\eta ) =\eta\\
\label{def-iso}& I(r)  = - 2 \Phi^{*} K(r) + K(r),
\end{align}
for any $X\in {\mathfrak X}(M)$,
$\eta \in \Omega^{1}(M)$, $r\in \Gamma (\mathcal G_{1})$. 
Above $\Phi^{*} : \mathcal G_{2} \rightarrow T^{*}M$ is defined by 
$$
(\Phi^{*} r) (X):= \langle  r, \Phi (X)\rangle_{\mathcal G_{2}},\ \forall  r\in \mathcal G_{2},\ X\in  {\mathfrak X}(M).
$$
The system $(K, \Phi , \beta)$ is subject to the following conditions:
\begin{align}
\nonumber&\nabla_{X}^{2} r = K \nabla^{1}_{X} ( K^{-1} r) + [r , \Phi (X)]_{\mathcal G_{2}}\\
\nonumber&  R_{2}(X, Y) = K R_{1}(X, Y) -(d^{\nabla^{2}} \Phi )(X, Y) -  [\Phi (X), \Phi (Y) ]_{\mathcal G_{2}}\\
\label{def-cond}& H_{2} = H_{1} - d\beta -  \langle (KR_{1} + R_{2} )\wedge \Phi\rangle_{\mathcal G_{2}} + c_{3}
\end{align}
where  $X, Y\in {\mathfrak X}(M)$, $r\in \Gamma (\mathcal G_{1})$
and
$$
 c_{3} (X, Y, Z):= \langle \Phi (X), [ \Phi (Y), \Phi (Z) ]_{\mathcal G_{2}} \rangle_{\mathcal G_{2}}.
 $$ 
 Conversely, if   $E_{1}  = \mathbb{T}M\oplus \mathcal G_{1} $ is a standard Courant algebroid with  defining data 
$(\nabla^{1} , R_{1}, H_{1})$, 
then  any system $(K, \Phi , \beta )$ where 
 $ K: \mathcal G_{1}\rightarrow  \mathcal G_{2} $ an isomorphism of quadratic Lie algebra bundles, 
$\Phi \in \Omega^{1} (M, \mathcal G_{2})$  and  $\beta \in \Omega^{2}(M)$,  defines  a new standard Courant algebroid $E_{2}$
with  quadratic Lie algebra bundle $\mathcal G_{2}$  and  defining data  $(\nabla^{2}, R_{2}, H_{2})$ 
related to $(\nabla^{1}, R_{1}, H_{1})$ 
by~(\ref{def-cond}). 
If  an isomorphism $I: E_{1} \rightarrow E_{2}$ is defined by $(K, \Phi , \beta )$, then its inverse is defined by
$(K^{-1}, - K^{-1} \Phi , - \beta )$ (a consequence of Proposition 6 of \cite{cortes-david-JSG}).\\

A  generalised almost complex structure on a Courant algebroid $E$ is a smooth, skew-symmetric,  field of endomorphisms
$\mathcal J \in \Gamma  (\mathrm{End}\, E)$ which satisfies $\mathcal J^{2} = -\mathrm{Id}.$  We say that $\mathcal J$ is integrable
(or is a generalised complex structure) if the space of sections of its  $(1,0)$-bundle 
$L = \{ u\in E^{\mathbb{C}}\mid  \mathcal J u  = i u\}$  is closed under the Dorfman bracket of $E$.  
This is equivalent to $N_{\mathcal J} =0$, where $N_{\mathcal J}\in \Gamma (\bigwedge^{2} E^{*}\otimes E)\cong
\Gamma (\bigwedge^{2} E^{*}\otimes E^{*})$  is the Nijenhuis tensor defined by 
$$
N_{\mathcal J} (u, v) := [ \mathcal J u , \mathcal Jv] - [u, v ] -\mathcal J ( [\mathcal J u , v] +[u, \mathcal J v]),
$$
for all $u,v\in \Gamma (E)$. 
It is skew-symmetric  in all three arguments and therefore a section of $\Gamma (\bigwedge^3E^*)$.
Any complex structure $J$ and symplectic
form $\omega$ define generalised complex structures on the untwisted generalised tangent bundle $\mathbb{T}M$, by 
$$
\mathcal J_{J} = \left( \begin{tabular}{cc}
$J$ & $0$\\
$0$ & $- J^{*}$
\end{tabular}\right),\quad \mathcal J_{\omega} = \left( \begin{tabular}{cc}
$0$ & $- \omega^{-1}$\\
$\omega$ & $0$
\end{tabular}\right)
$$
where $\omega^{-1} $ is the inverse of  the map $\omega : TM \rightarrow T^{*}M$ which assigns to a vector $X$ the covector $i_{X}\omega .$

\section{Components of generalised almost complex structures}\label{components-sect}
Let $E = \mathbb{T}M \oplus \mathcal G$ be a standard Courant algebroid with scalar product $\langle \cdot , \cdot \rangle .$

\begin{lem} \label{comp:lem} There is a natural bijection between generalised almost complex structures $\mathcal J$   on $E$ and 
data $J\in \Gamma (\mathrm{End}\, TM)$, $A\in \Gamma (\mathrm{End}\, \mathcal G)$, $B: T^*M \to TM$, $\nu : T^*M \to \mathcal G$,
$C: TM \to T^*M$, $\mu : TM \to \mathcal G$ subject to the skew-symmetry conditions 
\[ B^* =-B,\; C^*=-C,\; A^* =-A\] 
and the quadratic equations
\begin{eqnarray}
&&J^2+BC-\nu^*\mu = -\mathrm{Id}_{TM}\label{J^2:eq}\\
&&-\mu \nu^*-\nu \mu^* +A^2 = -\mathrm{Id}_{\mathcal G}\label{A2:eq}\\
&&JB-BJ^*-\nu^*\nu=0\label{JB:eq}\\
&&\mu B -\nu J^* + A \nu =0\label{omegaB:eq}\\
&&CJ-J^*C -\mu^*\mu=0\label{five:eq}\\
&&\mu J + \nu C + A\mu =0\label{six:eq},
\end{eqnarray}
where $J^* : T^*M \to T^*M$, $A^* : \mathcal G \to \mathcal G$,  $\nu^* : \mathcal G \to TM$, $\mu^* : \mathcal G \to T^*M$, $B^*: T^*M \to TM$ and $C: TM \to T^*M$  (identified with their trivial extensions to $E$)  are the metric adjoints with respect to the scalar product on $E$  (which for $B, C$ and $J$ coincide  
with the dual map): 
\begin{eqnarray}
&&
(J^*\xi)(X)=\xi (JX),\; (B^*\xi)(\eta ) = \xi (B\eta),\; (C^*X)(Y) = X(CY)\\
&&\langle A^*r,s\rangle =  \langle r, As\rangle,\;  \langle  \nu^*r , \xi \rangle = \langle r,\nu \xi \rangle,\; \langle \mu^* r,X\rangle = \langle r, \mu X\rangle
\end{eqnarray}
for all $X, Y\in TM$, $\xi, \eta \in T^*M$, $r, s \in \mathcal G$. 
\end{lem} 
\begin{proof} Given data $J, A, B, C, \mu, \nu$ as above, the corresponding generalised almost complex structure is given by 
\begin{equation}\label{form-J}
\mathcal J = \left( \begin{array}{ccc} 
J&B&-\nu^*\\
C&-J^*&-\mu^*\\
\mu & \nu&A
\end{array}
\right) 
\end{equation}
with respect to the decomposition $E= TM \oplus T^*M \oplus \mathcal G$. 
The skew-symmetry conditions are equivalent to $\mathcal J^* = -\mathcal J$. The quadratic equations are equivalent to $\mathcal J^2 =-\mathrm{Id}$. 
\end{proof}
Sometimes we will consider $B\in \Gamma (\bigwedge^2 TM)$ as a bivector and $C\in \Omega^2(M)=\Gamma (\bigwedge^2 T^*M)$ as a two-form.

\subsection{The Poisson structure}\label{Poisson-sect}

The full set of equations for the integrability of $\mathcal J$ in terms of components is stated in  Lemma \ref{integr-comp} and
Theorem \ref{integr-comp-thm}.  An important consequence is the existence of a Poisson 
bivector which underlies any generalised complex structure.  Another application of the integrability equations  will be  given in Section \ref{deg:Sec}.

\begin{prop}\label{poisson:prop} Let $\mathcal J $ be a generalised almost complex structure with components  $J, A, B, C, \mu , \nu .$ If $\mathcal J$ is integrable, then $B\in \Gamma ( {\bigwedge}^{2} TM)$ is a Poisson bivector.
\end{prop}
\begin{proof}   The claim follows from relation 
 (\ref{poisson}) (see  Lemma \ref{integr-comp}). More precisely, 
$$
\pi N_{\mathcal J} (\xi , \eta ) = \mathcal P_{B} (\xi , \eta ),
$$
where 
$$ 
\mathcal P_{B} (\xi , \eta ) := {\mathcal L}_{B (\xi  )}  (B\eta )  - B(\mathcal L_{B (\xi ) } \eta ) + B(\mathcal L_{B (\eta )} \xi ) - B d (B (\eta , \xi )),
$$
for any $\xi , \eta \in \Omega^{1}(M).$
When $\mathcal J$ is integrable, $\mathcal P_{B} =0$, i.e.\ $B$ is a Poisson bivector.  Here we are using that 
$\mathcal P_{B}$ coincides with the Schouten bracket $[[B, B]]$. To verify  this, it suffices to observe that $\mathcal P_{B}$ is tensorial and has 
the same components $\mathcal P_{B}^{ijk}=\sum_{(i:j:k)}\sum_{\ell} B^{\ell i}\partial_{\ell}B^{jk}$ as the Schouten bracket $[[B, B]]$ in local coordinates.
\end{proof}

\section{Computation of the data $(W, \mathcal D,  \sigma, \varepsilon )$}
\label{2nd_data:sec}
Recall \cite{darboux} that the $(1,0)$ bundle $L= \{ v -i\mathcal Jv \mid v\in E \}$ of a generalised almost complex structure $\mathcal J$  on 
a standard Courant algebroid $E= \mathbb{T}M\oplus \mathcal G$ with $\pi (L)$ of constant rank can be described in terms of data  $(W, \mathcal D,  \sigma, \varepsilon )$, where 
$W=\pi (L) \subset (TM)^\mathbb{C}$ is a (complex) subbundle, 
\begin{equation}\label{D:eq}\mathcal D = \pi_{\mathcal G} (\ker \pi|_L)\subset \mathcal G^{\mathbb{C}}\end{equation} 
is a maximally isotropic subbundle, $\sigma \in \Gamma (\mathrm{Hom}(W,  \mathcal G^{\mathbb{C}}))$ and $\varepsilon \in \Gamma (\bigwedge^2 W^*)$. More precisely, 
\begin{eqnarray} L &=& L(W, \mathcal D,  \sigma, \varepsilon )\nonumber\\
&=& \{ X+\xi + \sigma X + r \mid X\in W,\; \xi \in (T^*M)^\mathbb{C},\;r\in \mathcal D,\nonumber\\
&&
\xi (Y)= 2 \varepsilon (X,Y) -\langle \sigma (Y), \sigma (X) +2r\rangle\; \forall Y\in W\}\label{L_data:eq}\end{eqnarray} 
and the conditions  from Corollary 30 of \cite{darboux} (ensuring that $L\cap \bar L=\{ 0\}$) are satisfied.\

In this section we will express the data $(W, \mathcal D,  \sigma, \varepsilon )$ in terms of the 
components $J, A, B, C, \mu, \nu$ introduced in Lemma \ref{comp:lem}. 

\begin{thm} \label{comp_to_data:thm}Let $\mathcal J$ be a generalised almost complex structure  on 
$E$ with  components $J, A, B, C, \mu, \nu$ and $(1,0)$-bundle $L$. 
We assume that  $\pi (L)$, $B$ and $\nu$ have constant rank. 
Then the following data define the $(1,0)$ bundle $L=L(W, \mathcal D,  \sigma, \varepsilon )$ of $\mathcal J$:
\begin{enumerate}
\item 
\begin{eqnarray} W =  \mathrm{Rg} \, (\mathrm{Id} -i J) \oplus i\mathrm{Rg}\,  B \oplus i (\mathrm{Rg}\, \nu^*)_{0}, \label{Wdecomp:eq}
\end{eqnarray} 
where $(\mathrm{Rg}\, \nu^*)_{0}$ is a complement of $(\mathrm{Rg}\,  B) \cap \mathrm{Rg}\, \nu^*$ in 
 $\mathrm{Rg}\, \nu^*$;
 \item 
 \[ \mathcal D = \{ (\mathrm{Id}-iA)r -i \nu \xi \mid \xi \in T^*M,\; r\in \mathcal G,\; B\xi = \nu^* r\};\]
 \item The homomorphism $\sigma : W \to \mathcal{G}^{\mathbb{C}}$ is given by 
 \[ \sigma (X)= \frac12 (\sigma_1(X)-i\sigma_1(iX)),\quad X\in W,\]
 where,  in terms of the decomposition \eqref{Wdecomp:eq} of $W$, $\sigma_1  : W \to \mathcal{G}^{\mathbb{C}}$ is 
 given by 
\begin{eqnarray}
\nonumber\sigma_1 (X-iJX) &=& -i \mu X\\
\nonumber\sigma_1 (iv) &=& i\nu B^{-1} v\\
\label{sigma-1-def}\sigma_1 (iw) &=& (\mathrm{Id}-iA)(\nu^*)^{-1}w
\end{eqnarray} 
for all $X\in TM$, $v \in  \mathrm{Rg}\,  (B)$ and $w\in \mathrm{Rg}\,  (\nu )_0$. Here 
$B^{-1}$ stands for the inverse of the isomorphism $B : \mathcal C_B \to \mathrm{Rg}\, B\subset TM$, 
where  $\mathcal C_B \subset T^*M$ is complementary to $\ker B$. Similarly,
$(\nu^*)^{-1}$ is the inverse of the isomorphism $\nu^* : \mathcal C_{\nu^*} \to  \mathrm{Rg}\, \nu^*\subset TM$, where 
$\mathcal C_{\nu^*}\subset \mathcal G$ is complementary to $\ker \nu^*$.
\item  The $2$-form $\epsilon \in \Gamma  ({\bigwedge}^{2} W^{*})$ is given by
\begin{align}
\nonumber \epsilon (X, Y) & = \frac{1}{4} \left( \eta (X, Y) - i  \eta (iX , Y) \right) \\
\label{epsilon}& +\frac{1}{8} \langle \sigma_{1} (X) - i \sigma_{1} (iX) ,
\sigma_{1} (Y) - i \sigma_{1} (iY) \rangle ,
\end{align}
for any $X, Y\in W$, where $\eta \in  \Gamma ( W^{*} \otimes_{\mathbb{R}}  ( T^{*}M )^{\mathbb{C}}) = \Gamma\,  
(\mathrm{Hom}_{\mathbb{R}} (W, ( T^{*}M )^{\mathbb{C}})) $ is given by
\begin{align}
\nonumber& \eta (X- i JX) = - i CX\\
\nonumber& \eta ( iv) = - (\mathrm{Id}  + i J^{*} ) B^{-1} v\\
\nonumber& \eta (i w) = i \mu^{*} ( \nu^{*} )^{-1} w
\end{align}
for any $X\in TM$,  $v\in \mathrm{Rg}\, B$,  $w\in (\mathrm{Rg}\, \nu^*)_0$ and    $\eta (X, Y) := \eta (X) (Y)$ for any $X, Y$.
\end{enumerate} 
\end{thm}

\begin{proof}
The $(1,0)$ bundle $L= \{ v -i\mathcal Jv \mid v\in E \}$ of $\mathcal J$ is given by 
\begin{equation} \label{L_in_comp:eq}L = \left\{ \left. \left( \begin{array}{c} X- i (JX+B\xi -\nu^* r)\\
\xi -i (CX-J^*\xi -\mu^*r)\\
r-i(\mu X +\nu \xi + Ar) 
\end{array}\right) \right| X\in TM,\;\xi \in T^*M,\; r\in \mathcal G\right\}. \end{equation}
Projecting $L$ to $(TM)^{\mathbb{C}}$ we obtain 
\begin{equation}
W= \mathrm{Rg} \, (\mathrm{Id} -i J) \oplus i (\mathrm{Rg}\,  B + \mathrm{Rg}\, \nu^*),
\end{equation} 
which implies  (\ref{Wdecomp:eq}).  
Similarly, $\mathcal D$ is obtained by projecting $\ker \pi|_L$ to $\mathcal G^{\mathbb{C}}$, see \eqref{D:eq}. 
Recall \cite{darboux} that while $W$ and $\mathcal D$ are completely determined by $L$, the homomorphism 
$\sigma$ is unique only up to addition of $\gamma  \in \Gamma (\mathrm{Hom}(W,  \mathcal D))$ and the $2$-form 
$\varepsilon\in \Gamma (\bigwedge^2W^*)$ is uniquely determined once $\sigma$ is fixed. Due to Proposition 21 of 
\cite{darboux} and its proof, it is sufficient to check that $\sigma (Z) \in \mathcal D_Z := \pi_{\mathcal G} (\pi|_L)^{-1}(Z)$
for any $Z\in W$. To check this we begin by defining an $\mathbb R$-linear map $\tilde \sigma_1 : W \to L$ such that 
$\pi \tilde \sigma_1 (X) =X$ for all $X\in W$: 
\begin{eqnarray*}
\tilde \sigma_1 (X-iJX) &=& X-iJX -iCX -i \mu X\\
\tilde \sigma_1 (iv) &=& iv-B^{-1}v-iJ^*B^{-1}v+i\nu B^{-1} v\\
\tilde\sigma_1 (iw) &=& iw + i\mu^* (\nu^*)^{-1}w+ (\mathrm{Id}-iA)(\nu^*)^{-1}w, 
\end{eqnarray*} 
for any $X\in TM$, $v\in \mathrm{Rg}\, B$ and $w\in  (\mathrm{Rg}\,  \nu^{*} )_{0}$, compare \eqref{L_in_comp:eq}. Next we define a $\mathbb{C}$-linear section of $\pi|_L : L \to W$ by 
\[ \tilde{\sigma}(X) := \frac12 (\tilde\sigma_1 (X) -i\tilde\sigma_1(iX)),\quad X\in W.\]
The formulas for $\sigma_1$ and $\sigma$ are obtained by projecting $\tilde \sigma_1$ and $\tilde \sigma$ to 
$\mathcal G^{\mathbb{C}}$.  To compute $\varepsilon$ determined by $\sigma$ we remark 
 that  $\eta  (X) = \pi_{T^{*} M} \tilde{\sigma}_{1} (X)$ for any $X\in W$  and use 
the following general lemma.

\begin{lem}\label{compute-e} 
Let $\mathcal J$ be a generalised almost complex structure on $E$ with  $(1,0)$-bundle 
$L = L(W, \mathcal D , \sigma , \epsilon )$ such that $W = \pi (L)$ has constant rank. 
Then 
\begin{equation} \label{epsilon:eq}\varepsilon (X,Y) = \frac12 ((\pi_{T^*M} \tilde \sigma X)(Y)+\langle \sigma (X),\sigma (Y)\rangle),\quad X,Y \in W \end{equation}
for any complex linear section $\tilde \sigma$ of $\pi|_L$ such that $\pi_{\mathcal G}\circ \tilde \sigma = \sigma$. 
\end{lem}
\begin{proof} Note first that the right-hand side of \eqref{epsilon:eq} is independent of the choice of $\tilde \sigma$, since any two
choices differ by a $\mathbb{C}$-linear map $\lambda : W \to W^0$ with values in the annihilator $W^0  \subset (T^*M)^{\mathbb{C}}$ of $W$.
In fact, as a difference of two sections with the same $\mathcal G^{\mathbb{C}}$-projection, $\lambda$ takes values in 
$L\cap (T^*M)^{\mathbb{C}}$ which coincides with $W^0$.
From \eqref{L_data:eq} we see that 
\begin{equation}\label{star}
\varepsilon (X,Y)  = \frac12 (\xi (Y) +\langle \sigma (Y),\sigma (X) +2r\rangle )
\end{equation}
which is independent of the choice of 
$\xi$ and $r$ such that
$u=X+\xi +\sigma ( X) + r \in (\pi\vert_{L})^{-1}(X)$.  In fact, any two choices of $\xi + r$ differ by an element of 
\[ L\cap ((T^*M)^{\mathbb{C}} \oplus \mathcal D)= W^0 \oplus \mathcal D.\]
In particular,
we can take  $u:= \tilde \sigma (X)$. Since  $\pi_{\mathcal G} \circ \tilde \sigma = \sigma$, we deduce that $r=0$.  Also,  $\xi = \pi_{T^*M} \tilde \sigma (X)$.
Relation (\ref{star}) implies  \eqref{epsilon:eq}.
\end{proof} 
This finishes the proof of Theorem \ref{comp_to_data:thm}.
\end{proof}

\begin{rem} {\rm 
Note that when $B$ is surjective, the expression 
(\ref{L_in_comp:eq})
implies that the projection $\pi|_L : L \to (TM)^\mathbb{C}$ 
is also surjective and hence $\pi (L)$ has  automatically constant  rank as assumed in Theorem \ref{comp_to_data:thm}. 
The theorem then holds with $(\mathrm{Rg}\, \nu^*)_{0}=0$ without requiring the assumption that $\nu$ has constant rank. 
This setting will be studied in the next section.}
\end{rem}

\section{Non-degenerate generalised complex  structures}
\label{nondegGCS:sec}
In this section we define  a class of generalised complex structures called non-degenerate, for which
the equations in Lemma  \ref{comp:lem} can be completely solved.  We describe 
them  in Proposition~\ref{intBsymp:thm}
and  we classify them  up to isomorphisms in Theorem \ref{class_nondeg:thm}. 
When the Courant algebroid is exact and untwisted,
Proposition \ref{intBsymp:thm} reduces  to the description of generalised complex structures in terms of
Hitchin pairs (see \cite{crainic}).\

Let $E = \mathbb{T}M\oplus \mathcal G$ be a standard Courant algebroid with defining data $(\nabla , R, H)$ and quadratic 
Lie algebra bundle $(\mathcal G , [\cdot ,\cdot ]_{\mathcal G} , \langle \cdot , \cdot \rangle_{\mathcal G}).$ 

\begin{defn}\label{non-deg} A generalised almost complex structure $\mathcal J$   on  $E$ is called 
{\cmssl non-degenerate} if $B= \pi \circ \mathcal J|_{T^*M} : T^*M \to TM$ is bijective. 
\end{defn}

\begin{rem}{\rm
Lemma \ref{Binv:lemma} below shows that the property of being non-degenerate can be defined 
(by means of dissections) 
for generalised complex structures on arbitrary transitive Courant algebroids (not necessarily in standard form).}  
\end{rem}

\begin{prop}\label{nondeg-alg} There is a natural bijection between non-degenerate generalised almost complex structures $\mathcal J$   on 
$E$ and 
data $J\in \Gamma (\mathrm{End}\, TM)$, $\tilde A\in \Gamma (\mathrm{End}\, \mathcal G)$, $B: T^*M \to TM$ bijective, $\nu : T^*M \to \mathcal G$, 
subject to the skew-symmetry conditions 
\[ B^* =-B,\; \tilde A^* =-\tilde A\] 
and the quadratic equations
\begin{eqnarray}
&&\tilde A^2= -\mathrm{Id}\label{cxstrA:eq}\\
&&JB-BJ^*=\nu^*\nu .\label{nunustar:eq}
\end{eqnarray}
The components $A, C$ and $\mu$  in Lemma \ref{comp:lem}  are given by
\begin{eqnarray} 
\nonumber A&=& \tilde A -\nu B^{-1}\nu^*\\
\nonumber C&=& -B^{-1}(\nu^*\tilde A \nu B^{-1} + \mathrm{Id}) -J^*B^{-1}J\label{C:eq}\\
\label{comp-prime}\mu &=& \nu B^{-1}J - \tilde A \nu B^{-1}
.\end{eqnarray}
\end{prop}
\begin{proof} From \eqref{omegaB:eq} we obtain
\begin{equation}\label{add-0}
\mu = (\nu J^* -A\nu )B^{-1}
\end{equation}
and then from \eqref{J^2:eq}  
\begin{align} 
\nonumber C&=-B^{-1}+B^{-1}\nu^*\mu -B^{-1}J^2\\
\label{add-1}& =-B^{-1} +B^{-1}\nu^*(\nu J^* -A\nu )B^{-1}-B^{-1}J^2. 
\end{align}
The remaining  quadratic equations 
in Lemma \ref{comp:lem} reduce to \eqref{cxstrA:eq} and \eqref{nunustar:eq}.  More precisely, \eqref{cxstrA:eq}
is equivalent to \eqref{A2:eq}, \eqref{nunustar:eq}  coincides with  (\ref{JB:eq})
and the remaining equations \eqref{five:eq} and \eqref{six:eq} follow. Writing in 
(\ref{add-0}) and (\ref{add-1})
$A$ in terms of $\tilde A$ and using \eqref{nunustar:eq} several times we arrive at the last two equations in  \eqref{C:eq}. 
The skew-symmetry equations for $A, B, C$ reduce to those for $B$ and $\tilde{A}$.
\end{proof}

In order to investigate  the integrability of $\mathcal J$ we define 
a modified connection on $\mathcal G$ by 
\begin{equation}\label{mod-conn}
\tilde \nabla_X := \nabla_X + \mathrm{ad}_{\nu B^{-1}X},\quad X\in TM, 
\end{equation}
where $\mathrm{ad}:\mathcal G \rightarrow \mathrm{Der}_{\mathrm{sk}} (\mathcal G )$ denotes the adjoint representation of $\mathcal G$,
and a $2$-form $B_J\in \Omega^2(M)$ by 
\begin{equation}\label{B-J}
B_J(X,Y) := \frac12 (J \cdot B^{-1})(X,Y)= -\frac12 (B^{-1}(JX,Y)+B^{-1}(X,JY)).
\end{equation}

\begin{prop} \label{intBsymp:thm}Let $\mathcal J$ be a non-degenerate generalised almost complex structure determined by  the data $J$, $\tilde A$, $B$ and $\nu$ as in Proposition~\ref{nondeg-alg}.  Then  $\mathcal J$
is integrable if and only if the following conditions hold: 
\begin{enumerate}
\item $B^{-1}$ is a symplectic form. 
\item $\tilde A\in \Gamma (\mathrm{End}\, \mathcal G)$ is an  integrable 
complex structure, i.e.\ for any $p\in M$ the $(1,0)$-eigenspace of $\tilde A_p$
is closed under the Lie bracket of  $\mathcal G_p$. 
\item $\tilde \nabla \tilde A=0$. 
\item The covariant derivative of the $\mathcal G$-valued $1$-form $\nu B^{-1}$ satisfies 
\begin{equation}\label{r}
d^\nabla (\nu B^{-1}) + [\nu B^{-1}, \nu B^{-1}]_{\mathcal G} + R=0.
\end{equation}
\item The exterior derivative of the $2$-form $B_J$ satisfies 
\begin{equation} \label{h}
dB_J + H + \langle (R-\frac13  [\nu B^{-1}, \nu B^{-1}]_{\mathcal G})\wedge \nu B^{-1}\rangle_{\mathcal G} =0.
\end{equation}
\end{enumerate}
Above, the $\mathcal G$-valued $2$-form $[\nu B^{-1}, \nu B^{-1}]_{\mathcal G}$ is given by 
\[ [\nu B^{-1}, \nu B^{-1}]_{\mathcal G}(X,Y) := [\nu B^{-1}X, \nu B^{-1}Y]_{\mathcal G},\quad X,Y\in TM. \]
\end{prop}
\begin{proof}  We apply  Theorem \ref{comp_to_data:thm} to the special case when $B$ is invertible. 
Then   $W=(TM)^{\mathbb{C}}$,  $\mathcal D = \mathrm{graph}\, (-i\tilde A)$
and we claim that 
\begin{equation}\label{sigma-prel}
\sigma = \frac{1}{2} (  \mathrm{Id} + i \tilde{A} ) \nu B^{-1}.
\end{equation}
Relation  (\ref{sigma-prel}) follows from the next computation:  for any $X\in TM$ , 
\begin{align}
\nonumber&  2\sigma (X) = \sigma_{1} (X - i JX) +\sigma_{1} (iJX) - i\sigma_{1} ( iX)\\
\nonumber& = - i\mu X + i \nu B^{-1} JX + \nu B^{-1} X\\
\nonumber& = - i( \nu B^{-1} JX - \tilde{A} \nu B^{-1} X) +  i\nu B^{-1} JX + \nu B^{-1} X,\\
\nonumber& =  \nu B^{-1} X + i \tilde{A} \nu B^{-1} X,
\end{align}
where we used the definition  (\ref{sigma-1-def}) of $\sigma_{1}$ and the expression of $\mu$ given by the third relation in (\ref{comp-prime}).\

Recall now that $\sigma$ 
can be modified by addition of an arbitrary form $\gamma \in \Gamma ( \mathrm{Hom} (W, \mathcal D)).$ Since 
$\mathcal D = \mathrm{graph}\, (-i\tilde A)$, we can  (and will) replace $\sigma$ by the new homomorphism (again denoted by $\sigma$):
\begin{equation}\label{newsigma}
\sigma = \nu B^{-1} .
\end{equation}
To compute $\epsilon \in \Gamma ( {\bigwedge}^{2} (T^{*}M )^{\mathbb{C}})$ determined by  $\sigma$ we employ Lemma \ref{compute-e}, 
with $\tilde{\sigma} : (TM)^{\mathbb{C}}  \rightarrow L$ given by
\[ 
\tilde{\sigma }( X) = X +  i  B^{-1}X- J^{*} B^{-1} X +\nu B^{-1} X.
\] 
Let us check that $\tilde{\sigma}(X)\in L$ for all $X\in TM$, and thus for all $X\in (TM)^{\mathbb{C}}$.
From (\ref{L_in_comp:eq}) we see that $\tilde{\sigma }( X)\in L$ if and only if the pair $(\xi , r) := (- J^{*} B^{-1} X, \nu B^{-1}X)$ solves there the system 
\begin{eqnarray*}JX+B\xi -\nu^* r&=&0\\
CX-J^*\xi -\mu^*r &=& -B^{-1}X\\
\mu X +\nu \xi + Ar&=&0.
\end{eqnarray*}
Substituting the expression for $(\xi ,r)$ into these equations we see that they reduce to \eqref{JB:eq},  
the adjoint of \eqref{J^2:eq} and \eqref{add-0}, respectively. This proves  that $\tilde \sigma (X)\in L$.
Applying Lemma \ref{compute-e} and using  \eqref{nunustar:eq} now shows  that
\begin{equation}
\epsilon =\frac12 (B_J +i B^{-1}).
\end{equation}
By \cite[Proposition 35]{darboux} $\mathcal J$ is integrable if and only if
\begin{enumerate}
\item $\mathcal D \subset \mathcal G^{\mathbb{C}}$ is a bundle of subalgebras,
\item $\tilde\nabla$ preserves $\mathcal D$,
\item $R + d^\nabla \sigma + [\sigma, \sigma]_{\mathcal G}=0$
(because $\sigma =\nu B^{-1}$ and $R$ are real and $\mathcal D \cap \mathcal G=0$),
\item  \begin{equation} \label{2depsilon:eq}2d\varepsilon + H + \langle (d^\nabla \sigma ) \wedge \sigma\rangle_{\mathcal G} + 2 \langle R\wedge \sigma \rangle_{\mathcal G}  
+ 2 \langle [\sigma , \sigma ]_{\mathcal G},\sigma \rangle_{\mathcal G} =0,\end{equation} 
where 
\[ \langle [\sigma , \sigma ]_{\mathcal G},\sigma \rangle_{\mathcal G} (X,Y,Z) = \langle [\sigma (X) , \sigma (Y)]_{\mathcal G}, \sigma (Z)\rangle_{\mathcal G}, \quad X,Y,Z\in TM .\] 
\end{enumerate}
These conditions imply directly the conditions 2.-4.\ in the statement of Proposition  \ref{intBsymp:thm}. The remaining condition 1.\ follows by taking 
the imaginary part of \eqref{2depsilon:eq}  (or by using  Proposition \ref{poisson:prop} and the non-degeneracy of $B$)   
while condition 5.\ follows by taking the real part  of \eqref{2depsilon:eq} 
and rewriting it using that $d^\nabla \sigma = - R -[\sigma , \sigma ]_{\mathcal G}$, 
$\langle [\sigma , \sigma ]_{\mathcal G},\sigma \rangle_{\mathcal G} =
\frac13 \langle [\sigma , \sigma ]_{\mathcal G}\wedge \sigma \rangle_{\mathcal G}$ and $\sigma = \nu B^{-1}$.
\end{proof}

\begin{cor} Under the assumptions of Proposition \ref{intBsymp:thm}, the connection $\tilde \nabla$ is flat.
\end{cor}
\begin{proof}Recall that $R^\nabla =\mathrm{ad}_R$, which implies that 
\[ R^{\tilde \nabla}= \mathrm{ad}_{R+d^\nabla (\nu B^{-1}) +[\nu B^{-1}, \nu B^{-1}]_{\mathcal G}}=0.\]
\end{proof}

We now arrive at the main result from this section.

\begin{thm}\label{class_nondeg:thm}
Any non-degenerate generalised complex structure $\mathcal J$ on a transitive Courant algebroid is isomorphic (via Courant algebroid isomorphisms)
to a generalised complex structure  $\mathcal J_{\mathrm{can}} : E_{\mathrm{can}}\rightarrow
E_{\mathrm{can}}$, where
$$
E_{\mathrm{can}} = \mathbb{T}M\oplus
\mathcal G_{\mathrm{can}}
$$
is an untwisted Courant algebroid with quadratic Lie algebra bundle $\mathcal G_{\mathrm{can}} $ and (flat) connection  $\nabla^{\mathrm{can}}$, and
\begin{equation}\label{J-A}
\mathcal J_{\mathrm{can}} = \mathcal J_{\omega} \oplus A
\end{equation}
is the direct sum of a
complex structure  of symplectic type $\mathcal J_{\omega}$  on $M$ and  a
$\nabla^{\mathrm{can}}$-parallel  skew-symmetric  field
$A\in \Gamma ( \mathrm{End}\, \mathcal G_{\mathrm{can}} )$
of integrable complex structures on  $\mathcal G .$ 
Moreover, the symplectic structure $\omega$ is determined by the isomorphism type of $\mathcal J$   and
 its diffeomorphism type depends only on the equivalence class  (under equivalences of  Courant algebroids) of $\mathcal J$.
\end{thm}

\begin{proof} Let 
$I_{E} : E_{1} \rightarrow E_{2}$  be an isomorphism between standard Courant algebroids $E_{i}$, 
with quadratic Lie algebra bundles $(\mathcal G_{i}, 
[\cdot , \cdot ]_{\mathcal G_{i}}, \langle \cdot  , \cdot \rangle_{\mathcal G_{i}})$
and data $(\nabla^{i}, R_{i}, H_{i}).$  Assume that $I_{E}$ is 
defined by $(K, \Phi , \beta )$ (see relations (\ref{def-iso})) and  let 
$\mathcal J$ be a generalised almost complex structure on $E_{2}$, given by (\ref{form-J}).  

\begin{lem}\label{Binv:lemma}The $(J,A,B,\nu )$-components of the transformed generalised complex structure 
$\mathcal J_1=I_{E}^{-1} \circ \mathcal J \circ I_{E}$ are given by 
\begin{eqnarray}
J_{1} &=& J + B ( \beta -\Phi^{*}\Phi ) -\nu^{*} \Phi\label{J1:eq}\\
A_1 &=& K^{-1}\left( A + 2\Phi B\Phi^*-2\nu \Phi^* + \Phi \nu^* \right) K\\ 
B_{1} &=& B\\
 \nu_{1} &=& K^{-1} (  \nu -\Phi B  ) \label{nu1:eq} ,
\end{eqnarray}
where in the first relation $\beta$ is seen as a map from $TM$ to $T^{*}M$. 
In particular, $\mathcal J_1$ is non-degenerate if and only  if $\mathcal J$ is and, 
in that case, has the same underlying symplectic structure.
\end{lem}
\begin{proof}All  components of $\mathcal J_1$ can be obtained by a straightforward computation from (\ref{def-iso}) and (\ref{form-J}). 
For brevity, we have only listed a subset which fully encodes the generalised complex structure in the non-degenerate case. 
\end{proof}

Assume now that $E_{i}$ have the same quadratic Lie algebra bundle, which, for simplicity, will be denoted by 
$(\mathcal G , [\cdot , \cdot ]_{\mathcal G}, \langle \cdot , \cdot \rangle_{\mathcal G})$, 
and that 
$\mathcal J $ is  integrable and non-degenerate. Let $\tilde{\nabla } = \tilde\nabla^2$  be the modified connection  (\ref{mod-conn}) 
and $B_{J}$ the $2$-form defined by (\ref{B-J}), associated to $\mathcal J .$   Let 
$(K = \mathrm{Id},  \Phi  = \nu B^{-1}, \beta = B_{J})$. 
From relation (\ref{r})  combined with  the second relation in (\ref{def-cond}),  we obtain that
$R_{1} =0$. From the definition of $\tilde{\nabla}$  and  the first relation in (\ref{def-cond}),  we obtain that 
$\nabla^{1} = \tilde{\nabla }$. From relation (\ref{h}) 
combined with  the third relation in (\ref{def-cond}) and 
$$
c_{3} = \frac{1}{3} \langle [ \Phi , \Phi ]_{\mathcal G}\wedge \Phi \rangle_{\mathcal G}
$$
we obtain that $H_{1} =0.$ Finally, \eqref{nu1:eq} implies  that $\nu_{1}  =0.$\

We have proved that up to Courant algebroid isomorphisms,  any non-degenerate generalised complex structure  
$\mathcal J$ is defined on an untwisted Courant algebroid $E_{\mathrm{can}}$ and  its   $\nu$-component   is trivial
($\nu =0$).  In particular, 
the corresponding $2$-form $B_{J}$ is closed (from relation (\ref{h}) again) and 
$B_{J} = - B^{-1} \circ J$, in virtue of \eqref{nunustar:eq}.  
Let $I_{E_{\mathrm{can}}} \in \mathrm{Aut} (E_{\mathrm{can}})$ be the automorphism defined by
$( K =\mathrm{Id},  \Phi =0, B_{J})$. 
It is an easy check that 
$I_{E_{\mathrm{can}}}^{-1} \circ \mathcal J \circ I_{E_{\mathrm{can}}}$
is of the form (\ref{J-A}) with the symplectic form $\omega =C=-B^{-1}$.
Indeed, its  $J$-component is trivial,  in virtue of \eqref{J1:eq},
while its $\nu$-component  coincides with the  $\nu$-component  of $\mathcal J$, that is, it is trivial as well.
The statements that $A$ is integrable and $\nabla^{\mathrm{can}}$-parallel follow from the second and third conditions in Proposition
 \ref{intBsymp:thm}.
 The invariance of the symplectic structure $\omega = -B^{-1}$ under isomorphisms of Courant algebroids follows from Lemma~\ref{Binv:lemma}.
  In order to prove  the invariance of its diffeomorphism type under equivalences, it is sufficient to write any  equivalence 
 $\mathcal E : E_{1} \rightarrow E_{2}$ between two  untwisted Courant algebroids $E_{i}  = \mathbb{T}M \oplus \mathcal G_{i}$
 with quadratic Lie algebra bundles  $(\mathcal G_{i} , [\cdot , \cdot ]_{\mathcal G_{i}},
 \langle \cdot , \cdot \rangle_{\mathcal G_{i}})$ and connections $\nabla^{i}$, 
 which covers a diffeomorphism $f\in \mathrm{Diff} (M)$,  as a composition  $I\circ \mathcal E_{f}$, where $I : (f^{-1})^{!} E_{1} \rightarrow 
 E_{2}$ is an isomorphism and $\mathcal E_{f} : E_{1}  \rightarrow (f^{-1})^{!} E_{1}$  is an equivalence which covers $f$, defined by
 $$
 \mathcal E_{f} (X_{p}+ \xi_{p} + r_{p}) := (d_{p} f)(X_{p}) + \xi_{p}\circ (d_{p}f)^{-1} + r_{p},
 $$
 for any $X_{p}\in T_{p} M,\ \xi_{p} \in T^{*}_{p}M,\ r_{p} \in (\mathcal G_{1})_{p}$
(where in the right hand side $r_{p}$ is seen as a vector from $((f^{-1})^{*}  \mathcal G_{1} )_{f(p)} =( \mathcal G_{1})_{p}$). 
 Above we denoted by $(f^{-1})^{!} E_{1}$ the pullback Courant algebroid by the diffeomorphism $f^{-1}$, 
 which is an untwisted Courant algebroid with quadratic Lie algebra bundle  and connection the pullback  by $f^{-1}$ of $(\mathcal G_{1} , [\cdot , \cdot ]_{\mathcal G_{1}},
 \langle \cdot , \cdot \rangle_{\mathcal G_{1}})$ and $\nabla^{1}$ respectively
 (see \cite{bland}  for the definition of pullback Courant algebroid  and \cite[Section 4.2]{cortes-david-JSG} for more details  when the Courant algebroid is in standard form). 
Since $\mathcal E = I\circ \mathcal E_{f}$ and $I$ preserves the symplectic structure underlying any non-degenerate complex structure, while 
$\mathcal E_{f}$ preserves  its diffeomorphism type, 
the last statement  of the theorem follows.
 \end{proof}

Using the Riemann-Hilbert correspondence relating flat vector bundles and representations of the fundamental group, see e.g.\ \cite[Proposition 1.2.5]{K}  (and  \cite[Theorem 2.9]{morita} for the analogous  statement in the setting of principal bundles) 
we
reformulate Theorem
\ref{class_nondeg:thm} as follows.
(Below,  by an isomorphism $\phi : (J_\mathfrak{g}, \rho) \to (J_\mathfrak{g}', \rho ')$  between two pairs 
$(J_\mathfrak{g}, \rho)$ and  $(J_\mathfrak{g}', \rho ')$
we mean an automorphism of $(\mathfrak g , \langle \cdot , \cdot \rangle_{\mathfrak{g}})$ such that $\phi \circ J_\mathfrak{g}=J_\mathfrak{g}'\circ \phi$ and 
$\rho' = \phi \circ \rho$).

\begin{thm} \label{classif:cor} Let $(M, \omega )$ be a  symplectic manifold and  $(\mathfrak g , \langle \cdot , \cdot \rangle_{\mathfrak{g}})$ a quadratic Lie algebra.
There is a natural bijection between, on the one hand, isomorphism classes  of non-degenerate generalised complex structures with underlying symplectic structure $\omega$ on transitive Courant algebroids $E$ over $M$ with quadratic Lie algebra bundle of fiber type $(\mathfrak g , \langle \cdot , \cdot \rangle_{\mathfrak{g}})$ and, on the other hand,  isomorphism classes  of pairs $(J_\mathfrak{g}, \rho )$, where $J_\mathfrak{g}$ 
is a skew-symmetric integrable complex structure on $\mathfrak g$ and 
$\rho : \pi_1(M) \to \mathrm{Aut}(\mathfrak{g},\langle \cdot , \cdot \rangle_{\mathfrak{g}}, J_\mathfrak{g})$ is a representation by $J_{\mathfrak{g}}$-linear automorphisms of the quadratic Lie algebra. \end{thm}

\begin{proof}
Given a pair $(J_\mathfrak{g}, \rho )$ as above, we consider the flat bundle of quadratic Lie algebras $\mathcal G_{\mathrm{can}} :=\widetilde M\times_{\pi_1(M)} \mathfrak g$ with  its
canonical flat connection  $\nabla^{\mathrm{can}}$. The invariance of $J_{\mathfrak g}$ under $\rho (\pi_1(M))$ implies that $\mathcal G_{\mathrm{can}}$ carries a corresponding
$\nabla^{\mathrm{can}}$-parallel endomorphism field $A$, which is fiber-wise a skew-symmetric integrable complex structure. Then
$\mathcal J_{\mathrm{can}} = \mathcal J_{\omega} \oplus A$ defines
a non-degenerate generalised complex structure on  the untwisted Courant algebroid  $E_{\mathrm{can}}= \mathbb{b}TM \oplus  \mathcal G_{\mathrm{can}}$ determined by $\nabla^{\mathrm{can}}$ and $\mathcal G_{\mathrm{can}}.$  Conversely, it follows from Theorem \ref{class_nondeg:thm} (using the Riemann-Hilbert correspondence) that any non-degenerate generalised complex structure on a transitive Courant algebroid with symplectic form $\omega$ and fiber type
$(\mathfrak g , \langle \cdot , \cdot \rangle_{\mathfrak{g}})$ is isomorphic to one of the above form. 

 The fact that the above map $(J_\mathfrak{g}, \rho ) \mapsto
 \mathcal J_{\mathrm{can}}$ induces a bijection on isomorphism classes follows by first observing
 from  Lemma \ref{Binv:lemma}  (by setting $J_{1} = J =0$ and $\nu_{1} = \nu =0$) 
 and relations 
(\ref{def-cond}),  that any isomorphism $I$ 
 between two non-degenerate generalised complex structures $\mathcal J$, $\mathcal J'$ of the form \eqref{J-A} with corresponding endomorphisms $A, A'$ and 
 flat bundles $(\mathcal G, \nabla)$, $(\mathcal G',\nabla')$  is necessarily encoded by $(K, \Phi , \beta )$ with $\Phi=0$, $\beta=0$ and 
 $K\in \mathrm{Isom}(\mathcal G, \mathcal G')$  a parallel section mapping $A$ to $A'$,  
and then applying the Riemann-Hilbert correspondence between 
 isomorphism classes of flat vector bundles and isomorphism classes of representations of the fundamental group. Note that under this correspondence 
 $\nabla$-parallel sections correspond to tensors invariant under the holonomy representation $\rho$ of the flat connection $\nabla$.
In particular, $K$ defines an isomorphism between the pairs $(J_{\mathfrak{g}}, \rho )$ and $(J'_{\mathfrak{g}}, \rho^{'} )$.
\end{proof}

We now state various consequences of Theorems \ref{class_nondeg:thm}  and \ref{classif:cor}.  From Theorem \ref{classif:cor} we obtain:

\begin{cor}\label{reduction} Let $(M, \omega )$ and $(\tilde{M}, \tilde{\omega })$ be two symplectic manifolds with isomorphic fundamental groups. 
There is a  bijective correspondence between  isomorphism classes of non-degenerate generalised complex structures on transitive Courant algebroids  over $M$ and $\tilde{M}$, with underlying symplectic forms $\omega$ and $\tilde{\omega}$ respectively.
\end{cor}

\begin{exa}{\rm
Let $(M, \omega )$ be a compact  symplectic manifold with a Hamiltonian action of a compact
Lie group, such that the symplectic reduction  $(M_{\mathrm{red}}, \omega_{\mathrm{red}})$ at any coadjoint orbit in the image of the moment map  is smooth.
Then the fundamental groups of $M$ and   $M_{\mathrm{red}}$ are isomorphic
\cite[Theorems 1.2 and 1.3]{hui}   and Corollary \ref{reduction} can be applied to $(M, \omega )$ and $(M_{\mathrm{red}}, \omega_{\mathrm{red}})$.}
\end{exa}

\begin{cor}\label{consequences} Let $E$ be a transitive Courant algebroid  over $M$ with quadratic Lie algebra bundle $\mathcal G$ and fiber type
$(\mathfrak{g}, \langle \cdot , \cdot \rangle_{\mathfrak{g}}).$ Assume that $E$ admits a non-degenerate generalised complex structure. Let $E_{\mathrm{can}}$ be
the untwisted  Courant algebroid 
from Theorem \ref{class_nondeg:thm} 
and $\rho : \pi_{1}(M)\rightarrow \mathrm{Aut} ( \mathfrak{g}, \langle \cdot , \cdot
\rangle_{\mathfrak{g}}, J_{\mathfrak{g}}) $ the representation from  Theorem \ref{classif:cor}.\

i) Then $\mathfrak{g}$ admits a skew-symmetric  (integrable)  complex structure.\

ii) If  $\rho$ is trivial (i.e.\ $\rho (\pi_{1}(M)) = \mathrm{Id}_{\mathfrak{g}}$)  then $\mathcal G$ is trivial (as a vector bundle).\

iii)  If $\mathfrak{g} = \mathfrak{h}\oplus \mathfrak{c}$ is the direct sum of a Lie algebra $\mathfrak{h}$ without center and an abelian  Lie algebra $\mathfrak{c}$
(in particular,   if $\mathfrak{g}$ is reductive), 
then $E_{\mathrm{can}}$ is heterotic if and only if  $\mathrm{Im}\,  (\rho )\subset\mathrm{Int} (\mathfrak{g}).$  
\end{cor}

\begin{proof}  i) Claim i) follows  from Theorem \ref{class_nondeg:thm}.\

ii)  From the Riemann-Hilbert correspondence, $\rho$ is trivial if and only if $\mathcal G_{\mathrm{can}}$ is trivial and $\nabla^{\mathrm{can}}$ is the trivial connection. As $\mathcal G$ and $\mathcal G_{\mathrm{can}}$ are isomorphic as vector bundles, $\mathcal G$ is  trivial.\

iii)  If $E_{\mathrm{can}}$ is heterotic then  the connection $\nabla^{\mathrm{can}}$ of $\mathcal G_{\mathrm{can}}= P\times_{\mathrm{Ad}} \mathfrak{g}$ is induced by a  principal connection $\theta\in \Omega^1(P,\mathfrak g)$ on a principal $G$-bundle $P$  (where
$G$ is a  connected Lie group with Lie algebra $\mathfrak{g}$) and the holonomy group of $\nabla^{\mathrm{can}}$ is included in the adjoint group $\mathrm{Ad}_{G}.$ 
One statement in iii)  follows (without any assumptions on $\mathfrak{g}$). 
For the converse, assume that $\mathfrak{g} =\mathfrak{h} \oplus \mathfrak{c}$, where  is  $\mathfrak{c}$ is the center of $\mathfrak g$, and that
$\mathrm{Im}\, (\rho ) \subset \mathrm{Int} (\mathfrak{g})$.
Then $\rho = \rho_{\mathfrak{h}} \oplus \rho_{\mathfrak{c}}$ decomposes into the direct sum of  two representations, where 
$\rho_{\mathfrak{h}} : \pi_{1}(M) \rightarrow \mathrm{Aut} (\mathfrak{h})$ satisfies 
$\mathrm{Im}(\rho_{\mathfrak{h}}) \subset \mathrm{Int}(\mathfrak{h})$ 
and $\rho_{\mathfrak{c}}:\pi_{1}(M) \rightarrow \mathrm{Aut} (\mathfrak{c})$ is trivial.  Therefore, it is sufficient to prove the statement when $\mathfrak{g}$ has no center 
and, respectively,  when it is abelian. Assume first  that $\mathfrak{g}$ has no center  and let $G$ be a connected Lie group, with trivial center, with Lie algebra $\mathfrak{g}.$ Since $\mathrm{Im}\, (\rho ) \subset \mathrm{Int} (\mathfrak{g})$, the connection $\nabla^{\mathrm{can}}$ 
induces a  principal connection  $\theta$ on a subbundle  $P$ of the principal frame bundle of $E_{\mathrm{can}}$,  with structure group $\mathrm{Ad}_{G}.$
Since $G$ is isomorphic to $\mathrm{Ad}_{G}$, we can (and will)  consider $P$ as a principal $G$-bundle. 
Then   $\mathcal G_{\mathrm{can}}$ is the adjoint bundle of  $P$ and $\nabla^{\mathrm{can}}$ is induced by $\theta.$ Hence $E_{\mathrm{can}}$ is heterotic.
Finally, assume  that $\mathfrak{g}$ is abelian. 
Then  $\rho$ is trivial
and $E_{\mathrm{can}}$ is heterotic, with $\mathcal G_{\mathrm{can}}= M\times \mathfrak{g}$  associated to the trivial principal bundle  $\pi : M\times G\rightarrow M$
and canonical flat principal  connection (where $\mathrm{Lie}\, (G) = \mathfrak{g}$).
\end{proof}

In the next   examples the Courant algebroid is of neutral signature precisely when
$k = \ell $ and $\langle \cdot , \cdot \rangle_{\mathfrak{g}}$ is of neutral signature,
 respectively.

\begin{exa}\label{abelian:ex}{\rm   i) Consider the abelian Lie algebra
$\mathfrak g= \mathbb{R}^{2m}$ endowed with the pseudo-hermitian structure $(\langle \cdot , \cdot \rangle_{\mathfrak g}, J_{\mathfrak{g}})$ obtained from the standard
identification $\mathbb{C}^{k,\ell}\cong \mathbb{R}^{2m}$, $k+\ell =m$.
Let $\Gamma$ be the fundamental group of a symplectic manifold $(M,\omega )$ and $\rho : \Gamma \to \mathrm{U}(k,\ell )$ any representation.
Then $\rho$ determines a non-degenerate generalised complex structure $\mathcal J$  on a transitive Courant algebroid over $M$ with underlying
symplectic structure $\omega$ and with a bundle of abelian quadratic Lie algebras of signature $(2k,2l)$. The set of  isomorphism types
of such structures $\mathcal J$ is the quotient
\[ \mathrm{Hom}(\Gamma , \mathrm{U}(k,\ell ))/\mathrm{U}(k,\ell ),\]
where the action is by conjugation in $\mathrm{U}(k,\ell )$.
Explicit examples with non-trivial finite or infinite holonomy group are easy to give. We can choose, for instance, a symplectic torus $(M,\omega )$ and any non-trivial
homomorphism  $\pi_1(M) \to \mathbb{Z}/2\mathbb{Z} \cong \{ \pm \mathrm{Id}\} \subset \mathrm{U}(k,\ell)$.    As soon as $\rho$ is non-trivial, $E_{\mathrm{can}}$ 
is not heterotic.\

ii) Note that, given any group $\Gamma$, for the existence of a non-trivial representation $\rho : \Gamma \to \mathrm{U}(m,m)$ for some $m$, the existence of
any non-trivial representation of real dimension $m$ is sufficient. This follows from
the inclusions
\[ \mathrm{GL}(m, \mathbb{R}) \subset \mathrm{Sp}(2m,\mathbb{R}) \subset \mathrm{SU}(m,m). \]
The first inclusion follows from the linear symplectic identification $(\mathbb{R}^{2m}, \omega_{\mathrm{can}})\cong T^*\mathbb{R}^m = \mathbb{R}^m \oplus (\mathbb{R}^m)^*$, which shows
that $\mathrm{GL}(m, \mathbb{R})$ acts naturally by symplectic transformations. The second inclusion follows by considering the complexification
$(\mathbb{C}^{2m},  \Omega = \omega_{\mathrm{can}}^{\mathbb{C}})$  of $(\mathbb{R}^{2m}, \omega_{\mathrm{can}})$ with the standard
complex conjugation denoted by $\tau$ and\linebreak[4]  $\gamma := \sqrt{-1}\Omega (\cdot , \tau \cdot )$ the sesquilinear pseudo-hermitian form (of neutral signature) induced by the  data $(\Omega , \tau)$ on the complex vector space $\mathbb{C}^{2m}$.
Then we see that
\[ \mathrm{Sp}(2m,\mathbb{R})\cong \mathrm{Aut}(\mathbb{C}^{2m}, \Omega , \tau) \subset \mathrm{Aut}(\mathbb{C}^{2m},  \gamma ) \cong \mathrm{SU}(m,m). \] }
\end{exa}

\begin{exa}{\rm  Let $(M, \omega )$ be a symplectic manifold such that $\pi_{1} (M) = \mathbb{Z}^\ell$, $\ell \in \mathbb{N}$ (e.g.\ a symplectic torus).
Let $G$ be a compact Lie group of even rank with Lie algebra $\mathfrak{g}$ and
 $T$  a maximal torus of $G$ with Lie algebra $\mathfrak{t}$.  Define
 $$
 \mathfrak{g}^{(1,0)} := \mathfrak{t}^{(1,0)} + \mathfrak{g} (R^{+})
 $$
where $\mathfrak{t}^{(1,0)}$  is the $(1,0)$-space of a complex structure
$J_{\mathfrak{t}}$
on $\mathfrak{t}$ and $R^{+}$ is a set of positive roots
of the semi-simple Lie algebra $[\mathfrak{g}, \mathfrak{g}]^{\mathbb{C}}$ relative to   the Cartan subalgebra $(\mathfrak{t}\cap [\mathfrak{g}, \mathfrak{g}])^{\mathbb{C}}$. Then $\mathfrak{g}^{(1,0)}$ is the $(1,0)$-space of an  (integrable) complex structure
on $\mathfrak{g}$ (see  \cite{wang,sam}). 
Let  $\langle \cdot , \cdot \rangle_{\mathfrak{g}}$ be
any bi-invariant scalar product, such that $J_{\mathfrak{t}}$ is skew-symmetric with respect to the restriction of
 $\langle \cdot , \cdot \rangle_{\mathfrak{g}}$ to $\mathfrak{t}.$ ($J_{\mathfrak{t}}$ can be always chosen such that such scalar product exists.)
 Then $J_{\mathfrak{g}}$ is skew-symmetric with respect to 
 $\langle \cdot , \cdot \rangle_{\mathfrak{g}}$.
 Choose  $g_{i} \in T$, for $i\in \{ 1, \cdots , \ell\}$ and define
$\rho : \pi_{1} (M)\rightarrow \mathrm{Aut} (\mathfrak{g})$ by
\begin{equation}\label{repr}
\rho (k_{1}, \cdots , k_{\ell}) := \Pi_{i}
(\mathrm{Ad}_{g_{i}})^{k_{i}}.
 \end{equation}
 Then  $\rho : \pi_{1} (M) \rightarrow \mathrm{Aut} ( \mathfrak{g}, \langle \cdot , \cdot \rangle_{\mathfrak{g}} ,
J_{\mathfrak{g}})$ is a representation by $J_{\mathfrak{g}}$-linear automorphisms. Moreover,  $\rho$ is trivial if and only if  $g_{i} \in Z(G)$ (the center of $G$) for all $i$.
 The  Courant algebroid $E_{\mathrm{can}}$ is heterotic.}
\end{exa}

\begin{exa}{\rm
Let $\mathfrak{g}$ be a Lie algebra. Consider the non-semisimple Lie algebra $\mathfrak{g}_1=\mathfrak{g}\oplus \mathfrak{g}^{*}$, with Lie bracket
\begin{equation}
 [X+\xi , Y+\eta ] = [X, Y] - \eta \circ \mathrm{ad}_{X} +\xi \circ \mathrm{ad}_{Y}
\end{equation}
and ad-invariant natural pairing
$\langle\cdot , \cdot \rangle_{\mathfrak{g}\oplus \mathfrak{g}^{*}}$ of neutral signature. Skew-symmetric complex structures on $\mathfrak{g}_1$
(or, equivalently, left-invariant generalized complex structures  on $G$, where $G$ is any Lie group with Lie algebra $\mathfrak{g}$) were described in 
\cite[Theorem 6]{dmitri} 
in terms of so called admissible pairs   
$(\mathfrak{k}, \omega )$. An admissible pair \cite[Definition 5]{dmitri} consists of a subalgebra   
$\mathfrak{k}\subset \mathfrak{g}^{\mathbb{C}}$ 
with $\mathfrak{k}+ \bar{\mathfrak{k}} = \mathfrak{g}^{\mathbb{C}}$ (where bar denotes the complex conjugation of $\mathfrak{g}^{\mathbb{C}}$ with respect to the real form $\mathfrak{g}$)
and a closed $2$-form $\omega \in {\bigwedge}^{2} (\mathfrak{k}^{*})$ the imaginary part of which is non-degenerate on $\mathfrak k \cap \mathfrak g$.  The $(1,0)$-space of the corresponding skew-symmetric complex structure  on $\mathfrak{g}_{1}$ is given by
$$
L =  \{ X+ \xi  \in \mathfrak{k}\oplus  (\mathfrak{g}^{\mathbb{C}} )^{*} \; : \;\xi\vert_{\mathfrak{k}} = i_{X}\omega \} .
$$ 
Assume that $\mathfrak{g}$ is semisimple of inner type  and let $T$ be a maximal torus of $G$.
Let $R^{+}$ be a positive root system
with respect to the complexification
$\mathfrak{h}$ of $\mathfrak{t} = \mathrm{Lie}(T)$ and let
$\mathfrak{h}_{0}\subset\mathfrak{h}$ such that $\mathfrak{h}_{0} + \bar{\mathfrak{h}}_{0} = \mathfrak{h}.$
Then
$$
\mathfrak{k} = \mathfrak{h}_{0} + \mathfrak{g} (R^{+})
$$
together with any $2$-form $\omega_{0}\in {\bigwedge}^{2}
(\mathfrak{h}_{0}^{*})$ (trivially extended to
$\mathfrak{k}$), such that
$\mathrm{Im} \left( {\omega}_{0}\vert_{\mathfrak{l} }\right)$ is non-degenerate
(where $\mathfrak{l} := \mathfrak{h}_{0}
\cap \mathfrak{g}$), define an admissible pair and hence a skew-symmetric complex structure $J$ of $\mathfrak{g}_{1}$, which is preserved by the adjoint action of $T$.\
Other classes of admissible pairs (on semisimple Lie algebras of outer type), which are preserved by the adjoint action of a maximal torus of $\mathfrak{g}$,
can be constructed using Section 5 of \cite{dmitri}. For example, when $\mathfrak{g} = \mathfrak{sl}_{n}(\mathbb{H})$, 
we may take 
$$
\mathfrak{k} = \mathfrak{h}_{0} + {\sum_{i\neq j}}  \mathfrak{g}_{\epsilon_{ij}} +\sum_{i, j^{\prime}}  \mathfrak{g}_{\epsilon_{ij^{\prime}}} 
$$
and
$$
\omega = \omega_{0} + \sum_{i\neq j} \eta_{(ij)} \omega_{\epsilon_{ij}} \wedge \omega_{\epsilon_{ji}}
$$
where 
$\eta_{(ij)}\in \mathbb{C}$ are arbitrary constants, 
$\omega_{0}\in {\bigwedge}^{2}
(\mathfrak{h}_{0}^{*})$ and   $\omega_{\epsilon_{ij}}\in ( \mathfrak{g}_{\epsilon_{ij}})^{*} $ are trivially extended to $\mathfrak{k}$,
and the roots  $\epsilon_{ij}$ and $\epsilon_{ij^{\prime}}$ are given by 
$\epsilon_{ij} = \epsilon_{i} -\epsilon_{j}$, $\epsilon_{ij^{\prime}} = \epsilon_{i} - \epsilon_{j^{\prime}}$ 
(for the precise statements and notation, see \cite[Proposition 22]{dmitri}
and \cite[Theorem 23]{dmitri}). 
 As opposed to $\omega_{0}$ above, the action of 
 $\omega$ on  the root part of $\mathfrak{k}$ is non-trivial (unless $\eta_{(ij)} =0$ for all $i, j$).\

If $(M, \omega )$ is a symplectic manifold
and $J$ is the skew-symmetric complex structure on $\mathfrak{g}_{1}$ defined by such an admissible pair, then any morphism
$$
\rho : \pi_{1}(M)
\rightarrow\mathrm{Ad}^{G_1}(T)= (\mathrm{Ad}^{G}\oplus (\mathrm{Ad}^{G})^*)(T)\subset
\mathrm{Aut} ( \mathfrak{g}\oplus \mathfrak{g}^{*})
$$
defines a representation by $J$-linear automorphisms of the quadratic Lie algebra
$\mathfrak{g}_{1}$. 
Here $(\mathrm{Ad}^{G})^*$ denotes the co-adjoint representation  $(\mathrm{Ad}^{G})^*(a) = (\mathrm{Ad}^{G}(a^{-1}))^*$ for all $a\in G$.}
\end{exa}

\subsection{Obstructions on the fiber type}
In this section we present two classes of examples  of quadratic Lie algebras  $(\mathfrak{g}, \langle \cdot , \cdot \rangle_{\mathfrak{g}})$ which 
are not the fiber type of   the quadratic Lie algebra bundle of a  transitive Courant algebroid admitting a non-degenerate  generalised complex structure.   
This is a consequence of Corollary \ref{consequences} i), since the first example admits no complex structure, while the second example admits complex structures, but none of them is  skew-symmetric with respect to $\langle\cdot , \cdot \rangle_{\mathfrak{g}}$.

\begin{exa}\label{no_cx:ex}{\rm 
Consider the $6$-dimensional Lie algebra $\mathfrak{g}$ which has the following 
non-zero differentials in terms of a basis $\{ e^1,\ldots ,e^6\}$ of $\mathfrak{g}^*$:
\[ de^4=e^1 \wedge e^2,\quad de^5=e^1\wedge e^4,\quad de^6=e^2\wedge e^4.\]
According to Lemma \ref{7param:lem} below it has a $7$-dimensional space of 
$\mathrm{ad}$-invariant scalar products. Let $\langle \cdot , \cdot \rangle$ be any such scalar product. 
According to \cite{salamon}, $\mathfrak g$ does not admit any complex structure. (Note that $\mathfrak g$ is isomorphic to the Lie algebra 
$(0,0,0,12,14,24)$ in \cite{salamon}.) 
As a consequence, none of the quadratic Lie algebras $(\mathfrak{g}, \langle \cdot, \cdot \rangle)$ can occur as fiber type  
of the quadratic Lie algebra bundle $\mathcal G$ of a transitive Courant algebroid admitting a non-degenerate generalised complex structure.}
\end{exa}

\begin{lem}\label{7param:lem} Consider the Lie algebra of Example \ref{no_cx:ex}. The space $(\mathrm{Sym}^2\mathfrak g^*)^{\mathfrak g}$ of 
$ad$-invariant symmetric bilinear forms on $\mathfrak{g}$ is of dimension $7$ and its open subset 
consisting of scalar products is dense. An example 
of a scalar product of neutral signature is 
\[ \beta_0=2 e^1e^6-2e^2e^5 -(e^3)^2 +(e^4)^2. \] 
\end{lem}
\begin{proof}
We need to solve the linear system 
\[ \mathrm{ad}_{e_1}\cdot \beta = \mathrm{ad}_{e_2}\cdot \beta = \mathrm{ad}_{e_4}\cdot \beta =0,\]
where $(e_i)$ is the basis of $\mathfrak{g}$ dual to $(e^i)$, the unknown $\beta$ is a
symmetric bilinear form on $\mathfrak g$ and the dot denotes the natural action on the tensor algebra. 
Since the co-adjoint action is trivial on the subspace $\mathrm{span}\{ e^1,e^2, e^3\}$,   
the components $\beta_{ij}=\beta (e_i,e_j)=\beta_{ji}$ for $i,j \in \{1,2,3\}$ are $6$ free parameters. 
By analysing the linear system, we see that the remaining components are all equal to zero with exception of 
\[ \beta_{16}=-\beta_{25}=\beta_{44},\]
which is the seventh free parameter.  
The solution $\beta_0$ corresponds to $\beta_{16}=-\beta_{25}=\beta_{44}=-\beta_{33}=1$. 
Its non-degeneracy shows that the Zariski open subset of $(\mathrm{Sym}^2\mathfrak g^*)^{\mathfrak g}$ consisting 
of scalar products  is non-empty and therefore dense. 
\end{proof}

\begin{exa}\label{ex-2}{\rm
Consider the $6$-dimensional Lie algebra $\mathfrak{g}$ which has the following 
non-zero differentials in terms of a basis $\{ e^1,\ldots ,e^6\}$ of $\mathfrak{g}^*$:
\begin{equation} \label{000}
de^4=e^1 \wedge e^2,\quad de^5=e^1\wedge e^3,\quad de^6=e^2\wedge e^3.
\end{equation}
In the notation of \cite{salamon}, $\mathfrak{g}$ is the Lie algebra $(0,0,0,12,13,23)$.
Note that $\mathfrak{g}$ has an $\mathrm{ad}$-invariant  scalar product of neutral signature, given by 
$$
\rho (e_{i}, e_{j})= \rho(e_j,e_i) =( -1)^{i-1},\ \forall i\leq j,\ i+ j =7,
$$
and all other components equal to zero.
Also, $\mathfrak{g}$ admits a complex structure $J$  with the space of $(1,0)$-covectors given by 
$$
{\bigwedge}^{1,0} = \mathrm{span}_{\mathbb{C}}\{ e^{1} + i e^{2}, e^{3} + i e^{4}, e^{5} + i e^{6} \}  
$$
Note however that $J$ is not skew-symmetric with respect to $\rho .$
The following stronger statement  implies 
that $\mathfrak{g}$   (with any $\mathrm{ad}$-invariant scalar product) cannot be 
the fiber type of a transitive Courant algebroid admitting a non-degenerate
generalised complex structure.}
\end{exa}

\begin{lem}\label{lem-ex}
i) Consider the Lie algebra of Example \ref{ex-2}. The space $(\mathrm{Sym}^2\mathfrak g^*)^{\mathfrak g}$ of 
$ad$-invariant symmetric bilinear forms on $\mathfrak{g}$ is of dimension $7$ and its open subset 
consisting of scalar products is dense.\

ii) There is no complex structure  on $\mathfrak{g}$ which is skew-symmetric with respect to an $\mathrm{ad}$-invariant scalar product. 
\end{lem}

\begin{proof} i) If $\tilde{\rho}$ is an $\mathrm{ad}$-invariant symmetric bilinear form, then we can write it as $\tilde{\rho }(u, v) = \rho (Au, v)$ where $A\in \mathrm{End} (\mathfrak{g})$ satisfies
\begin{equation}\label{ad-sim}
\mathrm{ad}_{u} \cdot A =0,\ \rho (Au, v) = \rho (u, Av),
\end{equation}
for any $u, v\in \mathfrak{g}.$  
Let $(e_{i})$ be the basis of $\mathfrak{g}$ dual to the basis $(e^{i}).$  
From the  first condition in (\ref{ad-sim})  together with (\ref{000}), we deduce that 
$A(e_{i}) = \lambda e_{i}$ for any $4\leq i\leq 6$ and 
\begin{equation}
A (e_{i}) = \lambda e_{i} +\sum_{j=4}^{6} \lambda_{ij} e_{j}
\end{equation}
for any $1\leq i\leq 3$, where $\lambda , \lambda_{ij}\in \mathbb{R}$ are arbitrary. From the  second condition in (\ref{ad-sim}) we deduce that 
$\lambda_{26} = -\lambda_{15}$, $\lambda_{36} = \lambda_{14}$ and $\lambda_{35} = - \lambda_{24}.$  Letting $\alpha := \lambda_{14}$,
$\beta := \lambda_{15}$ and $\gamma := \lambda_{24}$,  we obtain that 
$\tilde{\rho} (e_{i}, e_{j}) = 0$ for any  $i\le j$, except 
\begin{align}
\nonumber& \tilde{\rho} (e_{1}, e_{1}) =\lambda_{16},\  \tilde{\rho}( e_{2}, e_{2})  = -\lambda_{25},\ \tilde{\rho}( e_{3}, e_{3}) = \lambda_{34}\\
\nonumber& \tilde{\rho} (e_{1}, e_{2}) = -\beta ,\ \tilde{\rho} (e_{1}, e_{3}) = \alpha ,\ \tilde{\rho}( e_{1}, e_{6}) = \lambda \\
\label{tilde-rho}& \tilde{\rho} ( e_{2}, e_{3}) =\gamma ,\   \tilde{\rho} (e_{3}, e_{4}) =  - \tilde{\rho}( e_{2}, e_{5}) =\lambda .
\end{align}
Thus  $\tilde{\rho}$
depends on  7  independent parameters 
and it is  non-degenerate if and only if $\lambda \neq 0.$\

ii)  It is sufficient to prove  that for any complex  structure  $J$ on $\mathfrak{g}$, there is a basis $(e^{i})$ of $\mathfrak{g}^{*}$ which satisfies 
(\ref{000})
(together with $de^{i} =0$ for any $i\in \{1,2,3\}$) 
 such  that ${\bigwedge}^{1,0}_{J}$ is generated by  forms $\omega^{1}$, $\omega^{2}$ and $\omega^{3}$, where 
\begin{equation}\label{omega-13}
\omega^{1} = e^{1} + a e^{2},\  \omega^{3} = e^{5} + b e^{6}
\end{equation}
($a, b\in \mathbb{C}\setminus \{ 0\}$)  and   
\begin{equation}\label{omega-2}
\omega^{2} \in \mathrm{span}_{\mathbb{C}} \{ e^{i},\ 1\leq i\leq 4\}
\end{equation}
satisfies  
$d\omega^{2} \in {\bigwedge}^{2,0}_{J} \oplus {\bigwedge}_{J}^{1,1}$.   
Once we know this,  we deduce that   there is a $(1,0)$-vector of the form 
$v:= e_{3} -\mu e_{4}$ (with $\mu \in \mathbb{C}\setminus \mathbb{R}$). But the latter  is not isotropic with respect to any 
$\mathrm{ad}$-invariant scalar product   (from  (\ref{tilde-rho}) with $\lambda \neq 0$).

The  above statement is  implicitly  contained in  \cite[Theorem 3.2]{salamon}. 
For completeness of our exposition we briefly explain  how it follows from  the arguments in  \cite{salamon}. 
In a first stage, one shows that for any
$6$-dimensional nilpotent Lie algebra $\mathfrak{h}$ with a complex structure $J$, there is a basis $(f^{i})$ of $\mathfrak{h}^{*}$
which satisfies 
\begin{equation}\label{prop-d}
d f^{i} \in {\bigwedge}^{2} \langle f^{1}, \cdots , f^{i-1} \rangle ,\ 1\leq i\leq 6
\end{equation}
(where  $f^{0} := 0$
and $\langle f^{1}, \cdots , f^{i-1} \rangle$ denotes the vector space spanned by $f^{i}, \cdots , f^{i-1}$) 
such that 
$df^{6}$ does neither involve $f^{3}\wedge f^{5}$ nor $f^{4}\wedge f^{5}$, 
${\bigwedge}^{1,0}_{J}$ is generated either by
\begin{equation}\label{sit1}
\omega^{1} = f^{1} + i f^{2},\ \omega^{2} = f^{3} + i f^{4},\ \omega^{3} = f^{5} + i f^{6}
\end{equation}
or by 
 \begin{equation}\label{sit2}
\omega^{1} = f^{1} + i f^{2},\ \omega^{2} = f^{4} + i f^{5},\ \omega^{3} = f^{3} + i f^{6}
\end{equation}
and   the following equations are satisfied:
\begin{equation}\label{conditions-s}
d\omega^{1} = 0,\  \omega^{1}\wedge  d\omega^{2} = 0,\ \omega^{1}\wedge \omega^{2} \wedge d\omega^{3} =0
\end{equation}
(see \cite[Theorem 2.5]{salamon}). 
By  suitable base changes,  one obtains 
the classification of $6$-dimensional  nilpotent Lie algebras admitting a complex structure, stated in Theorems 3.1, 3.2 and 3.3 of \cite{salamon}.   Our Lie algebra $\mathfrak{g} = (0,0,0,12, 13,23)$ is obtained by such a procedure starting from
(\ref{sit1}).  (Starting from \eqref{sit2} leads to Lie algebras which are not isomorphic to $\mathfrak{g}$ and thus  \eqref{sit2} will not be considered here.)  More precisely, the  second relation in (\ref{conditions-s}) 
implies  that $df^{4}$ is a multiple of $f^{1}\wedge f^{2}$. The same holds for $df^{3}.$   Thus, there is a  new basis $(\tilde{f}^{i})$,
still satisfying (\ref{prop-d}) and, moreover,
$$
d\tilde{f}^{3} = 0,\ d\tilde{f}^{4} = \tilde{f}^{1}\wedge \tilde{f}^{2}.
$$
The covectors 
$\tilde{f}^{3}, \tilde{f}^{4} $ are linear combinations of  $f^{3}, f^{4}$ and $\tilde{f}^{i} = f^{i}$ for any $i\in \{ 1,2,5, 6\}$. 
 After rescaling $\omega^{2}$ if necessary we obtain in the new basis 
\begin{equation}\label{sit-tilde}
\omega^{1} = \tilde{f}^{1} + i \tilde{f}^{2},\ \omega^{2} = \tilde{f}^{3} + a \tilde{f}^{4},\ \omega^{3} = \tilde{f}^{5} + i \tilde{f}^{6}
\end{equation}
where $a\in \mathbb{C}\setminus \mathbb{R}$.  The third relation in (\ref{conditions-s}) implies  that
$d\tilde{f}^{6} \in {\bigwedge}^{2} \langle \tilde{f}^{1}, \tilde{f}^{2}, \tilde{f}^{3}, \tilde{f}^{4} \rangle$. From 
(\ref{prop-d}),  the same holds for
$d\tilde{f}^{5}.$    
Next, one considers various cases according to the vanishing or nonvanishing of 
the $4$-forms $d\tilde{f}^{5}\wedge d\tilde{f}^{5}$ and $d\tilde{f}^{6}\wedge d\tilde{f}^{6}$. 
The case 
$$
d\tilde{f}^{5}\wedge d\tilde{f}^{5} = d\tilde{f}^{6}\wedge d\tilde{f}^{6} = d\tilde{f}^{5}\wedge d\tilde{f}^{6} =0
$$
leads to  the  Lie algebra 
$\mathfrak{g}$  (and   to two other  Lie algebras with $b_{1} =3$, 
see  the proof of \cite[Theorem 3.2]{salamon}).  More precisely, one considers a suitably chosen 
linear transformation $A$ which preserves the subspaces   generated by $\tilde{f}^{1}$ , $\tilde{f}^{2}$, 
by $\tilde{f}^{5}$,   $\tilde{f}^{6} $, and  by  $\tilde{f}^{1}, \cdots  , \tilde{f}^{4}$ respectively, 
such that in the new basis  $e^{i} := A(\tilde{f}^{i})$ relations 
(\ref{000})  are satisfied and  $d e^{i} =0$ for any $1\leq i\leq 3$.  (For the two other isomorphism types of Lie algebras mentioned above, 
a similar transformation exists which brings the differentials to a different form. This is excluded in our setting because the Lie algebra was fixed from the beginning).

From the listed properties of $A$, we see that in the basis $(e^{i})$,  
$\omega^{1}$, $\omega^{2}$ and $\omega^{3}$ are 
of the form (\ref{omega-13}) and (\ref{omega-2}).  
Moreover all complex structures on $\mathfrak{g}$ enjoy this property 
(according to \cite[Apendix]{salamon}, 
$\mathfrak{g}$ is the unique  $6$-dimensional nilpotent Lie algebra with $b_{1} =3$ and $b_{2} =8$ which admits a complex structure.
In particular,  all other Lie algebras arising from the above argument are non-isomorphic to $
\mathfrak{g}$).  
\end{proof}

\section{Examples degenerating along a complex hypersurface}
\label{deg:Sec}
In this section we construct examples of generalised complex structures on transitive 
Courant algebroids which are non-degenerate precisely on the complement of a (possibly singular) complex  
analytic hypersurface. 
\begin{prop} \label{poisson_ex:prop}Let $(M,J, \beta )$ be a complex manifold endowed with a holomorphic Poisson structure $\beta$, 
$(\mathfrak{g}, \langle \cdot , \cdot \rangle_{\mathfrak{g}})$ a quadratic Lie algebra endowed with a skew-symmetric complex structure $J_\mathfrak{g}$ 
and $\rho : \pi_1(M) \to \mathrm{Aut}(\mathfrak g, \langle \cdot , \cdot \rangle_{\mathfrak{g}},  J_{\mathfrak{g}}  )$ a representation. 
Then these data define a generalised complex structure $\mathcal J$ on the untwisted Courant algebroid $E= \mathbb{T}M \oplus 
\mathcal G$  with the flat bundle of quadratic Lie algebras $\mathcal G = \widetilde{M} \times_{\rho} \mathfrak g$.
The generalised complex structure $\mathcal J$ has 
components $J$, $B=\beta + \bar\beta$ and $A$ and  all other components equal to zero. Here 
$A$ is the parallel field of skew-symmetric complex structures on the fibers of $\mathcal G$ associated with $J_\mathfrak{g}$. 
Moreover,  $\mathcal J$ is non-degenerate  precisely on the complement of the 
hypersurface $\{ \det \beta=0\}$, where $\beta$ is considered here as a map ${\bigwedge}^{1,0}T^*M \to T^{1,0}M$ (such that 
$\det \beta$ is a holomorphic section of $\det (T^{1,0}M)^{\otimes 2}$, the square of the anti-canonical line bundle). 
\end{prop}
\begin{proof}For the data $(J, B=\beta + \bar\beta, A)$ with all other components set to zero, the 
equations in Lemma \ref{comp:lem} reduce to $B^*=-B$, $A^*=-A$, $J^2=-\mathrm{Id}_{TM}$, 
$A^2=-\mathrm{Id}_{\mathcal G}$ and $JB = BJ ^*$, which are all satisfied. Note that the 
last equation is equivalent to $B$ being of type $(2,0)+(0,2)$ with respect to $J$. This proves 
that the above data define a generalised almost complex structure $\mathcal J$. 

Next we check 
its integrability using Theorem \ref{integr-thm}. Due to $C=0$, $\mu=0$, $\nu=0$, $H=0$ and $R=0$, the 
relations enumerated in Lemma~\ref{integr-comp} 
as 2, 3, 6 and 12 are all trivially satisfied. 
In the following we discuss the remaining relations, always with the vanishing components 
$C, \mu , \nu $, as well as $H, R$,  set to zero.
The relation listed as 1 is equivalent (with the vanishing components set to zero)
to the integrability of $J$ while 9 corresponds to the integrability of $A$. The relation 7 is equivalent 
to $\nabla A=0$ and implies 8, using that the complex structure $A$ is skew-symmetric and thus isometric.
This leaves us with the relations 4 and 5. The  relation 4 is satisfied if and only if $B$ is a Poisson tensor, as discussed
in the proof of Proposition \ref{poisson:prop}. The first two terms in the Schouten bracket $[[B,B]]=[[\beta , \beta ]] + [[\bar\beta , \bar\beta]] + 2[[ \beta , \bar \beta]]$ vanish
because $\beta$ is a (holomorphic) Poisson tensor and the third term vanishes because it is the bracket of a holomorphic with an anti-holomorphic bivector. 
This proves that $B$ is a Poisson tensor. The (tensorial) relation 5 can be easily checked by taking $\xi$ and $\eta$ holomorphic or anti-holomorphic 
(all possible combinations) and 
observing that $B$ maps holomorphic (respectively anti-holomorphic) covector fields to holomorphic (respectively anti-holomorphic)
vector fields. 
This proves the integrability of $\mathcal J$. The claim about the 
non-degeneracy set of  $\mathcal J$ follows from the fact that  the rank of $B_p : T_p^*M \to T_pM$ coincides with twice the complex rank of $\beta : {\bigwedge}^{1,0}T_p^*M \to T_p^{1,0}M$, $p\in M$.
\end{proof}
\begin{exa}{\rm Consider the Hopf surface $M= (\mathbb{C}^2\setminus \{ 0\})/\Gamma\cong S^3\times S^1$ with its canonical complex structure $J$, 
where $\Gamma\cong \mathbb{Z}$ denotes the cyclic group generated by the transformation $\mathbb{C}^2\setminus \{ 0\} \to \mathbb{C}^2\setminus \{ 0\}, z\mapsto \lambda z$, 
($\lambda \in \mathbb{R}\setminus \{ -1, 0, 1\}$). The scale invariant holomorphic Poisson structure $z^1z^2\partial_1\wedge \partial_2$ on $\mathbb{C}^2$ induces
a holomorphic Poisson structure $\beta$ on $(M,J)$ with the degeneracy locus $\{ z^1z^2=0\}\subset M$. Let $\rho : \pi_1(M)\cong \mathbb{Z} \to \mathrm{U}(k,\ell )$ be any non-trivial representation. 
Then the data $(M,J,\beta , \rho )$ define a generalised complex structure as in Proposition \ref{poisson_ex:prop} with a non-trivial bundle
of abelian quadratic Lie algebras. 
}
\end{exa}

\section{Appendix: integrability in terms of components}

Let $E = \mathbb{T}M\oplus \mathcal G$ be a standard Courant algebroid, with quadratic Lie algebra bundle 
$(\mathcal G , [\cdot , \cdot ]_{\mathcal G}, \langle \cdot , \cdot \rangle_{\mathcal G} )$ and data $(\nabla , R ,H).$
For simplicity,  $\langle \cdot , \cdot \rangle_{\mathcal G}$  will be denoted by $\langle \cdot , \cdot \rangle$.  

\begin{lem}\label{integr-comp} A generalised almost complex structure  $\mathcal J$   on $E$  with components $J, A, B, C, \mu , \nu $
is integrable if and only if the following relations hold:\

\begin{enumerate}
\item for any $X, Y\in {\mathfrak X}(M)$,
\begin{align}
\nonumber& {\mathcal L}_{JX} (JY) - \mathcal L_{X} Y - J \mathcal L_{X} (JY) + J\mathcal L_{Y} (JX)\\
\nonumber& +  B (  i_{X}  i_{JY} H +  i_{JX} i_{Y} H +i_{X} i_{Y} (dC) -  C\mathcal L_{X} Y )  \\
\nonumber& + 2B ( \langle i_{X}R, \mu (Y) \rangle 
-  \langle i_{Y}R, \mu (X) \rangle )\\
\label{unu}& + \nu^{*} \left(R(X, JY) - R(Y, JX) +\nabla_{X} (\mu Y) -\nabla_{Y} (\mu X) \right) =0
\end{align} 
\item for any $X, Y\in  {\mathfrak X}(M)$,
\begin{align}
\nonumber& i_{JY} i_{JX} H - i_{Y} i_{X} H  + J^{*} \left( i_{JY} i_{X} H - i_{JX} i_{Y}H \right) \\
\nonumber& + 2\langle i_{JY} R, \mu X\rangle - 2\langle i_{JX} R, \mu Y \rangle 
 - 2 J^{*} \left( \langle i_{X}R, \mu Y\rangle - \langle i_{Y}R, \mu X\rangle \right)\\
\nonumber& +\mathcal L_{JX} (C) Y -\mathcal L_{JY} (C) X 
+ d ( C(X, JY)) + J^{*} \left(i_{Y} i_{X} (dC) + C\mathcal L_{X} Y \right)\\
\nonumber& + 2 \langle \nabla (\mu X), \mu Y\rangle  +\mu^{*} \left( R(X, JY) +  R(JX, Y) +\nabla_{X} (\mu Y) -\nabla_{Y} (\mu X) \right)\\
\nonumber& = 0;
\end{align}\item for any $X, Y\in {\mathfrak X}(M)$,
\begin{align}
\nonumber&  R(JX, JY)  - R(X, Y) +\nabla_{JX} (\mu Y) -\nabla_{JY} (\mu X)  + [\mu X, \mu Y]_{\mathcal G} \\
\nonumber& -\mu ( \mathcal L_{X} (JY) -\mathcal L_{Y} (JX) )\\
\nonumber& +  \nu \left(i_{X} i_{JY}H +  i_{JX} i_{Y} H +i_{X} i_{Y} (dC) -  C\mathcal L_{X} Y\right)\\
\nonumber& +  2 \nu \left( \langle i_{X}R, \mu Y\rangle  -  \langle i_{Y}R, \mu X\rangle \right)\\
\nonumber& - A\left( R(X, JY) - R(Y, JX) +\nabla_{X} (\mu Y) -\nabla_{Y} (\mu X) \right)  =0;
\end{align}
\item for any $\xi , \eta \in \Omega^{1}(M)$,
\begin{equation}\label{poisson}
\mathcal L_{B\xi } (B\eta ) - B\left( \mathcal L_{B\xi } \eta - \mathcal L_{B\eta } \xi \right) + Bd  ( B(\xi , \eta ) ) =0;
\end{equation}
\item for any $\xi , \eta \in \Omega^{1}(M)$,
\begin{align}
\nonumber&\mathcal L_{B\eta } (J^{*}\xi )  - \mathcal L_{B\xi } (J^{*}\eta ) + d  (B(J^{*}\xi , \eta ) )
+ J^{*} \left( \mathcal L_{B\xi } \eta - \mathcal L_{B\eta } \xi - d  ( B(\xi , \eta ) )\right) \\
\nonumber& - 2\left( \langle i_{B\xi } R, \nu \eta  \rangle - \langle i_{B\eta } R, \nu \xi  \rangle \right) + 2 \langle \nabla (\nu \xi ), \nu \eta  \rangle
 + i_{B\eta} i_{B\xi } H =0;
\end{align}
\item for any $\xi , \eta \in \Omega^{1}(M)$,
\begin{align}
\nonumber& R(B\xi , B\eta )  + \nabla_{B\xi } (\nu \eta ) -\nabla_{B\eta } (\nu \xi ) +  [\nu \xi , \nu \eta ]_{\mathcal G}\\
\nonumber& - \nu \left( \mathcal L_{B\xi } \eta - \mathcal L_{B\eta } \xi  -  d ( B(\xi  , \eta )) \right) =0; 
\end{align}
\item for any $r, \tilde{r}  \in \Gamma (\mathcal G)$,
\begin{align}
\nonumber& {\mathcal L}_{\nu^{*} r} ( \nu^{*} \tilde{r})  - 2 B \left( \langle i_{\nu^{*} r} R, \tilde{r} \rangle 
-  \langle i_{\nu^{*} \tilde{r}} R, {r} \rangle  +\langle ( \nabla A)r,  \tilde{r} \rangle \right)\\
\nonumber& + \nu^{*} \left( [Ar, \tilde{r} ]_{\mathcal G} - [A\tilde{r}, r]_{\mathcal G} -\nabla_{\nu^{*} r} \tilde{r} +\nabla_{\nu^{*} \tilde{r}} r\right) 
=0;
\end{align}

\item for any $r, \tilde{r}  \in \Gamma (\mathcal G)$,
\begin{align}
\nonumber& i_{\nu^{*} \tilde{r}} i_{\nu^{*} r} H +\mathcal L_{\nu^{*}r} (\mu^{*} \tilde{r}) - \mathcal L_{\nu^{*}\tilde{r}} (\mu^{*}{r})
+ 2  d \langle \mu^{*}r, \nu^{*}\tilde{r} \rangle \\
\nonumber& + 2\left( \langle i_{\nu^{*} r} R, A\tilde{r}\rangle -  \langle i_{\nu^{*} \tilde{r}} R, A{r}\rangle 
+\langle \nabla  (Ar ),  A \tilde{r} \rangle - \langle \nabla r, \tilde{r} \rangle  \right)\\
\nonumber& + 2 J^{*} \left( \langle i_{\nu^{*} r} R, \tilde{r}\rangle  - \langle i_{\nu^{*} \tilde{r}} R, {r}\rangle 
+\langle ( \nabla A) r , \tilde{r} \rangle \right)\\
\nonumber& + \mu^{*} \left( [Ar, \tilde{r} ]_{\mathcal G} + [r, A\tilde{r} ]_{\mathcal G} -\nabla_{\nu^{*} r} \tilde{r} +\nabla_{\nu^{*}\tilde{r}} r
\right) =0;
\end{align}
\item for any $r, \tilde{r}  \in \Gamma (\mathcal G)$,
\begin{align}
\nonumber&  [Ar, A\tilde{r} ]_{\mathcal G}  -[ r, \tilde{r}]_{\mathcal G} - A\left( [Ar, \tilde{r} ]_{\mathcal G} + [r, A\tilde{r} ]_{\mathcal G}  \right)\\
\nonumber& - ( \nabla_{\nu^{*} r}  A) \tilde{r} +  ( \nabla_{\nu^{*} \tilde{r}}  A ) {r} +  R (\nu^{*} r, \nu^{*}\tilde{r} ) \\
\nonumber& - 2\nu \left( \langle i_{\nu^{*}r}R, \tilde{r} \rangle - \langle i_{\nu^{*} \tilde{r}} R, r\rangle +
\langle  (\nabla A) r, \tilde{r} \rangle \right) =0;
\end{align}
\item  for any $X\in {\mathfrak X}(M)$ and $\xi \in \Omega^{1} (M)$, 
\begin{align}
\nonumber& \mathcal L_{JX} (B\xi ) - B( \mathcal L_{JX} \xi ) - J \mathcal L_{X} (B\xi )  +  B\mathcal L_{X} (J^{*}\xi )\\ 
\nonumber& + B\left(i_{X}  i_{B\xi } H +  2\langle i_{X}R, \nu \xi  \rangle \right) + \nu^{*}  \left(\nabla_{X} (\nu \xi ) + R(X, B\xi ) \right) =0;
\end{align}
\item for any $X\in {\mathfrak X}(M)$ and $\xi \in \Omega^{1} (M)$, 
\begin{align}
\nonumber& i_{B\xi } i_{JX} H 
+ J^{*} i_{B\xi } i_{X} H  -\mathcal L_{JX} (J^{*}\xi )  - \mathcal L_{X} \xi+ J^{*}\left(  \mathcal L_{JX} \xi - \mathcal L_{X} (J^*\xi ) \right) \\
\nonumber& + i_{B\xi } i_{X} (dC)
- \mathcal L_{X} (i_{B\xi } C ) 
+ 2 (\langle i_{B\xi } R, \mu X\rangle - \langle i_{JX}R, \nu \xi \rangle )\\
\nonumber&  + 2\langle \nabla (\mu X) , \nu \xi \rangle  
  - 2 J^{*} \langle i_{X} R, \nu \xi \rangle  +\mu^{*}\left(  \nabla_{X} (\nu \xi ) + R(X, B\xi  ) \right)\\
  \nonumber&  =0;
\end{align}
\item for any $X\in {\mathfrak X}(M)$ and $\xi \in \Omega^{1} (M)$, 
\begin{align}
\nonumber&  \nabla_{JX} (\nu \xi ) - \nu \mathcal L_{JX} \xi 
-\nabla_{B\xi } (\mu X) -\mu \mathcal L_{X} (B\xi )+[\mu (X), \nu (\xi ) ]_{\mathcal G}\\
\nonumber& - \nu \left( i_{B\xi } i_{X} H - {\mathcal L}_{X} (J^{*}\xi ) - 2\langle i_{X} R, \nu \xi \rangle \right) - A \left( \nabla_{X} (\nu \xi ) + R(X, B\xi ) \right)\\
\nonumber& +R(JX, B\xi ) =0;
\end{align}
\item for any $r\in \Gamma (\mathcal G)$  and $\xi \in \Omega^{1} (M)$,
\begin{equation}
(\mathcal L_{\nu^{*} r} B) \xi  + 2B\left( \langle i_{B\xi } R, r\rangle +\langle \nabla  r, \nu \xi \rangle \right) +  \nu^{*} \left(\nabla_{B\xi } r - [r, \nu \xi ]_{\mathcal G} \right) =0;
\end{equation}
\item for any $r\in \Gamma (\mathcal G)$  and $\xi \in \Omega^{1} (M)$, 
\begin{align}
\nonumber& \mathcal L_{B\xi }
( \mu^{*} r)+ d ( B (\mu^{*} r, \xi )) -\mu^{*}  \nabla_{B\xi } r\\
\nonumber& +  i_{\nu^{*} r} i_{B\xi } H + ({\mathcal L}_{\nu^{*}r} J^{*})  \xi  + 2\langle i_{\nu^{*} r} R,  \nu \xi \rangle 
+ 2\langle i_{B\xi } R, Ar\rangle + 2\langle \nabla (Ar), \nu \xi  \rangle \\
\nonumber& + 2 J^{*} \left( \langle i_{B\xi } R, r\rangle +\langle \nabla r, \nu \xi \rangle\right)    + \mu^{*} 
[r, \nu \xi ]_{\mathcal G}  =0;
\end{align}
\item for any $r\in \Gamma (\mathcal G)$  and $\xi \in \Omega^{1} (M)$, 
\begin{align}
\nonumber& R (\nu^{*}r, B\xi ) +\nabla_{\nu^{*}r} (\nu \xi )  + ( \nabla_{B\xi } A) r -[Ar, \nu \xi ]_{\mathcal G} + A [r, \nu \xi ]_{\mathcal G}  \\
\nonumber& +  2\nu \left( \langle i_{B\xi } R, r\rangle +\langle \nabla r, \nu \xi \rangle 
- \frac{1}{2} \mathcal L_{\nu^{*}r} \xi \right)  
=0;
\end{align}

\item for any $r\in \Gamma (\mathcal G)$ and $X\in {\mathfrak X}(M)$,
\begin{align}
\nonumber& \mathcal L_{JX} (\nu^{*} r)   -  
J \mathcal L_{X} ( \nu^{*} r) -  2B\left(  \langle i_{JX} R, r\rangle  + \langle i_{X}R, Ar\rangle \right)   \\
\nonumber& - \nu^{*} \left( \nabla_{JX} r  + \nabla_{X} (Ar) + [\mu X, r]_{\mathcal G} - R( X, \nu^{*} r) \right)\\
\nonumber&  - B\left( i_{\nu^{*} r} i_{X} H  + \mathcal L_{X} (\mu^{*} r)  - 2 \langle \nabla ( \mu X), r\rangle  \right) =0;
\end{align}
\item for any $r\in \Gamma (\mathcal G)$ and $X\in {\mathfrak X}(M)$,
\begin{align}
\nonumber& i_{JX} i_{\nu^{*} r} H - J^{*}  i_{\nu^{*} r} i_{X} H   -\mathcal L_{JX} (\mu^{*} r) 
- J^{*}  \mathcal L_{X} (\mu^{*} r)\\
\nonumber&   - 2\langle i_{JX} R, Ar\rangle + 2 \langle i_{X} R, r\rangle  -  2 J^{*} (  \langle i_{X} R, Ar\rangle   +   \langle i_{JX} R, r\rangle  ) 
\\
\nonumber& +  i_{X} i_{\nu^{*}r} (dC) +\mathcal L_{X} (C\nu^{*} r) 
 - 2 \langle i_{\nu^{*} r} R, \mu X\rangle + 2\langle \nabla (\mu X), Ar\rangle \\
\nonumber&   +  2J^{*}  \langle \nabla (\mu X)  , r\rangle + \mu^{*} \left( \nabla_{JX} r +  \nabla_{X} (Ar)  + [\mu X, r]_{\mathcal G} - R(X, \nu^{*} r) \right)\\
\nonumber&   =0;
\end{align}

\item for any $r\in \Gamma (\mathcal G)$ and $X\in {\mathfrak X}(M)$,
\begin{align}
\nonumber& R (\nu^{*}r, JX)  + 2\nu ( \langle i_{X} R, Ar\rangle + \langle i_{JX} R, r\rangle ) - A R(\nu^{*} r, X) \\
\nonumber& +(\nabla_{JX} A) r -\nabla_{X} r- A\nabla_{X}(Ar) \\
\nonumber& +\nabla_{\nu^{*} r} ( \mu X) - \mu (\mathcal L_{\nu^{*} r} X) 
+[\mu X, Ar]_{\mathcal G} 
 - A  [\mu X, r]_{\mathcal G}  \\
\nonumber& + \nu \left( i_{\nu^{*}r} i_{X} H +\mathcal L_{X} (\mu^{*} r) - 2\langle \nabla (\mu X) , r\rangle\right)  =0.
\end{align}
\end{enumerate} 
\end{lem}

\begin{proof} The claim   follows  by projecting  the equality $N_{\mathcal J} (u, v) =0$
to $TM$, $T^{*}M$ and $\mathcal G$,  by letting $u$ and $v$ be sections of 
$TM$, $T^{*}M$ and $\mathcal G$, and by using the expression of $\mathcal J$ in terms of components.
In the first, second and third relations we also use 
\begin{equation}\label{c}
\mathcal L_{X} (CY) - \mathcal L_{Y} (CX) + d \left( C(X, Y)\right) = i_{Y} i_{X} (dC) + C(\mathcal L_{X} Y ),
\end{equation}
for any $X, Y\in {\mathfrak X}(M)$, 
in  the 7th, 8th and 9th relations we use
\begin{equation}
\langle \nabla (Ar), \tilde{r} \rangle - \langle \nabla (A\tilde{r}), r\rangle + d \langle A\tilde{r}, r\rangle =
\langle \nabla (A) r, \tilde{r}\rangle,
\end{equation}
for any $r, \tilde{r} \in \Gamma (\mathcal G)$, 
in the 11th relation we use again (\ref{c}) with $Y := B\xi $ 
and in the 17th relation we use
\begin{equation}
 \mathcal L_{\nu^{*} r} (CX) - d \left( C(X, \nu^{*} r)\right) + C \mathcal L_{X} (\nu^{*} r)
= i_{X} i_{\nu^{*} r} (dC) +\mathcal L_{X} (C\nu^{*} r),
\end{equation}
for any $X\in {\mathfrak X}(M)$,  $r\in \Gamma (\mathcal G)$.
\end{proof}

\begin{thm}\label{integr-comp-thm}
\label{integr-thm} A generalised almost complex structure  $\mathcal J$ with components $J, A, B, C, \mu , \nu $
is integrable if and only if   relations 1-9 and 12
 in Lemma \ref{integr-comp} hold. 
\end{thm}

\begin{proof}    In view of the decomposition
\begin{align}
\nonumber& {\bigwedge}^{3} E^{*} = {\bigwedge}^{3} (TM)^{*} \oplus {\bigwedge}^{3} (T^{*}M)^{*} \oplus {\bigwedge}^{3} \mathcal G^{*}\\
\nonumber&  \oplus {\bigwedge}^{2} (TM)^{*} \wedge (T^{*}M)^{*} \oplus {\bigwedge}^{2} (TM)^{*} \wedge \mathcal G^{*} \oplus {\bigwedge}^{2} (T^{*}M)^{*} 
\wedge (TM)^{*} \\\
\nonumber& \oplus  {\bigwedge}^{2} (T^{*}M)^{*}\wedge \mathcal G^{*} \oplus {\bigwedge}^{2}\mathcal G^{*} \wedge (TM)^{*} \oplus {\bigwedge}^{2} \mathcal G^{*}
\wedge (T^{*}M)^{*} \\
\nonumber& \oplus (TM)^{*} \wedge (T^{*}M)^{*} \wedge \mathcal G^{*},
\end{align}
and considering $N_{\mathcal J}$ as a map $E\times E \to E$ by means of the scalar product $\langle \cdot , \cdot \rangle$ we obtain that  $\mathcal J$ is  integrable if and only if the components
\begin{align}
\nonumber&  \pi_{T^{*}M}   N_{\mathcal J}\vert_{ TM\times TM},\ \pi N_{\mathcal J}\vert_{ T^{*}M\times T^{*}M},\    
 \pi_{\mathcal G} N_{\mathcal J}\vert_{\mathcal G\times \mathcal G},\, 
 \pi N_{\mathcal J}\vert_{TM\times TM},\\
 \nonumber&  \pi_{\mathcal G} N_{\mathcal J}\vert_{TM\times TM},\  \ \pi_{T^{*}M}   N_{\mathcal J}\vert_{ T^*M\times T^*M},\
 \pi_{\mathcal G}  N_{\mathcal J} \vert_{T^{*}M\times  T^{*}M},\ 
 \pi_{T^{*}M} N_{\mathcal J}\vert_{ \mathcal G\times \mathcal G},\\
\nonumber&  \pi N_{\mathcal J} \vert_{ \mathcal G \times\mathcal G},\ 
\pi_{\mathcal G} N_{\mathcal J}\vert_{ TM\times T^{*}M},
\end{align}
vanish, or, equivalently, the  relations from the statement of the theorem hold. 
Therefore,   the 18 relations  in  Lemma~\ref{integr-comp}
reduce to the 10 relations  from the statement of the theorem.  
\end{proof}
\begin{rem}{\rm Although the relations 10-11 and 13-18 in Lemma~\ref{integr-comp} follow from the other equations
we decided to keep them in the paper for future use. }
\end{rem}

V. Cort\'es: vicente.cortes@math.uni-hamburg.de\

Department of Mathematics and Center for Mathematical Physics, University of Hamburg,  Bundesstrasse 55, D-20146, Hamburg, Germany.\\

L. David: liana.david@imar.ro\

Institute of Mathematics  `Simion Stoilow' of the Romanian Academy,   Calea Grivitei no.\ 21,  Sector 1, 010702, Bucharest, Romania.

\end{document}